\documentclass[a4paper,11pt]{amsart}
\usepackage{amsmath,amssymb,amsthm,amsfonts}
\usepackage{color}
\usepackage[all]{xy}

 \usepackage{hyperref}

\newcommand{\EE}{\mathcal{E}}
\newcommand{\LL}{\mathcal{L}}
\newcommand{\FF}{\mathcal{F}}
\newcommand{\GG}{\mathcal{G}}
\newcommand{\HH}{\mathcal{H}}
\newcommand{\kk}{\mathbb{K}}

\newcommand{\ZZ}{\mathbb{Z}}

\newcommand{\CC}{\mathbb{C}}
 
\newcommand{\OO}{\mathcal{O}}

\newcommand{\VV}{\mathcal{V}}

\DeclareMathOperator{\coker}{coker}

\DeclareMathOperator{\Hom}{Hom}
\DeclareMathOperator{\Art}{\bf{Art}}
\DeclareMathOperator{\Grpds}{\bf{Grpds}}
\DeclareMathOperator{\Set}{\bf{Sets}}
\DeclareMathOperator{\Spec}{Spec}
\DeclareMathOperator{\Id}{Id}

\DeclareMathOperator{\MC}{MC}
\DeclareMathOperator{\Def}{Def}
\DeclareMathOperator{\Del}{Del}

\DeclareMathOperator{\Mor}{Mor}

\DeclareMathOperator{\Pic}{Pic}
\DeclareMathOperator{\Tot}{Tot}

\newcommand{\h}{\mathfrak{h}}
\newcommand{\g}{\mathfrak{g}}
\newcommand{\m}{\mathfrak{m}}

\newtheorem*{theorems}{Theorem}
\newtheorem{theorem}{Theorem}[section]
\newtheorem{corollary}[theorem]{Corollary}

\newtheorem{proposition}[theorem]{Proposition}
\newtheorem{lemma}[theorem]{Lemma}
\theoremstyle{definition}
\newtheorem{definition}[theorem]{Definition}
\newtheorem{remark}[theorem]{Remark}
\newtheorem{example}[theorem]{Example}

\definecolor{rosso}{RGB}{162,0,0}
\definecolor{verde}{RGB}{0,100,0}
\definecolor{blu}{RGB}{0,0,162}

\title{Deformations of morphisms of sheaves}
\author{Donatella Iacono}
\address{\newline Dipartimento di Matematica,
\newline  Universit\`a degli Studi di Bari Aldo Moro,
\hfill\newline Via E. Orabona 4,
70125 Bari, Italy.}
\email{donatella.iacono@uniba.it}

\author{Elena Martinengo}
\address{
\newline
Dipartimento di Matematica ``Giuseppe Peano'',
\hfill\newline
 Universit\`a degli Studi di Torino,
 \hfill\newline
 via Carlo Alberto 10,
10123 Torino, Italy}
\email{elena.martinengo@unito.it}


\begin{document}
\maketitle 

\begin{abstract}
We analyse infinitesimal deformations of morphisms of locally free sheaves on  a smooth projective variety $X$ over an algebraically closed field of characteristic zero. In particular, we describe a differential graded Lie algebra controlling the deformation problem. As an application, we study infinitesimal deformations of the pairs given by a locally free sheaf and a subspace of its sections with a view towards  Brill-Noether theory.
\end{abstract}

\tableofcontents
\addtocontents{toc}{\protect\setcounter{tocdepth}{1}}

 \section*{Introduction} 
Let $X$ be a smooth projective variety over  an algebraically closed field $\mathbb{K}$ of characteristic zero. Let $\FF$ and $\GG$ be locally free sheaves of $\OO_X$-modules on $X$ and $\alpha: \FF \to \GG$   a morphism of sheaves. 
In this article, we are interested in the infinitesimal deformations of $\alpha: \FF \to \GG$,  where both the sheaves $\FF$ and $\GG$ and the map $\alpha$ can be deformed.
This is equivalent to deform the graph of $\alpha$ as a subsheaf of the direct sum $\FF \oplus \GG$, in such a way that the deformation of  $\FF \oplus \GG$ is given by a deformation of $\FF$ and a deformation of $\GG$. 

Not much is known about them. We tackle this problem using  differential graded Lie algebras (dgLas).
In \cite{Dona.Tesi, Dona.def maps}, the first author investigated  infinitesimal deformations of a holomorphic map gluing the dgLas that control the deformations of  the domain, of the codomain and of the graph in the product.
Here, we apply the same approach to   deformations of a morphism of sheaves. 
Moreover, we extend it to any algebraically closed field $\kk$ of characteristic zero, using semicosimplicial dgLas techniques as developed in  \cite{FMM, FIM}.

The sheaves of dgLas associated to the   infinitesimal  deformations of a sheaf and of the graph inside the direct sum are classically known, each of them forms a semicosimplicial dgLa $\g^\Delta$ and the \v{C}ech functor $H^1_{sc}(\g^{\Delta})$ controls the corresponding deformations. 
A deformation of the morphism $\alpha: \FF \to \GG$ is obtained from the data of the three involved deformations via a totalisation process. Finally, applying the strong tool of 
Hinich's Theorem of descent of Deligne groupoids  \cite{hinich}, 
 we get a specific differential graded Lie algebra that controls the    infinitesimal  deformations of $\alpha: \FF \to \GG$ via the Deligne functor in groupoids. This dgLa $\Tot(H(\VV)^\Delta)$ is the Thom-Whitney dgLa associated to a suitable semicosimplicial dgLa $H(\VV)^\Delta$ (see Section \ref{sezione deformazione mappe} for details).  

\begin{theorems}[Theorem \ref{thm.Del eq Def}]

The functor $\Def_{(\FF, \alpha, \GG)}$ of    infinitesimal  deformations of the morphism $\alpha: \FF \to \GG$ is equivalent to the Deligne functor
$\Del_{\Tot(H(\VV)^\Delta)}$ associated to the Thom-Whitney dgLa of the semicosimplicial dgLa $H(\VV)^\Delta$.
\end{theorems}

This approach holds for any algebraically closed field $\kk$ of characteristic zero. If we restrict our attention to the field of complex numbers,  we can consider the dgLas associated to  the   Dolbeault resolutions of the relevant sheaves, instead of the \v{C}ech semicosimplicial dgLas and the Thom-Withney construction, as explained in Remark \ref{rmk.sui numeri complessi}.

\smallskip

As an application of our study of the   infinitesimal  deformations of a morphism of sheaves, we analyse the infinitesimal deformations of a locally free sheaf $\EE$ over a smooth projective variety $X$ over $\kk$ together with a linear subspace $U\subseteq H^0(X,\EE)$ of its global sections. Such a pair $(\EE,U)$ is called a coherent system. The study of the moduli space  and deformations of coherent systems is very classical and it is a wide and still very active research field. We refer the reader to the introduction of  \cite{DE Brill-Noether} for more details and references. 
The idea behind our approach is to view any    infinitesimal  deformation of the pair $(\EE,U)$  as an   infinitesimal deformation of the morphism $s: U\otimes \OO_X \to \EE$ of sheaves of $\OO_X$-modules,   induced by the inclusion $U\subseteq H^0(X, \EE)$. For this correspondence, we need the technical assumption that $H^1(X, \OO_X)=0$,  so that the sheaf $U\otimes \OO_X$ has only trivial infinitesimal deformations.  To study  the  infinitesimal  deformations of $s: U\otimes \OO_X \to \EE$, we can directly apply the previous construction to find a suitable semicosimplicial dgLa, whose Thom-Whitney dgLa controls infinitesimal deformations of the pair $(\EE, U)$ via the Deligne functor.

Moreover,  under the assumption that $H^i(X, \OO_X)=0$ for $i=1,2$,
its cohomology groups fit into an exact sequence (see Equation \ref{eq.theorem}), from which we obtain much information about the tangent and the obstructions spaces of the deformation problem of $(\EE, U)$ and a link with the Petri map.
The dgLa we described is quite  involved, but its importance lies on the fact  that it holds over any algebraically closed field  $\mathbb{K}$ of characteristic zero.

In our article  \cite{DE Brill-Noether}, we generalised many  classical results concerning coherent systems of line bundles over a curve, 
to the case of a vector bundle  of  any rank on a smooth projective variety  of   any dimension over $\CC$. 
  Here,  we recover these results over  any algebraically closed field $\kk$ of characteristic zero, under the additional hypothesis that $H^i(X, \OO_X)=0$ for $i=1,2$, see Corollaries \ref{corollario alpha su implica r liscio}, \ref{corcor}, \ref{prop.sezione si estende se cup product va a zero} and \ref{Cor. TDef(E,H0(E))}. This additional hypothesis do not allow us  to generalise  the results obtained over  $\CC$, for curves of genus $g>0$.

The Thom-Whitney dgLa constructed in this way is unfortunately quite complicated  to be handled. For this reason, we even present a more explicit model of a semicosimplicial dgLa whose total complex controls deformations of the pair $(\EE, U)$  (see Section \ref{sezione dgla esplicit model}).

\smallskip

The connection between  the  moduli space of coherent systems and Brill-Noether theory is obviously very close. The classical Brill-Noether theory concerns the subvarieties $W_d^k(C)$  of $\Pic^d(C)$ of line bundles on a curve $C$ of degree $d$ having at least $k+1$ independent  sections and   much is classically known about it.
During the last years, several generalisations of this problem were investigated. The general case of vector bundles  on varieties of higher dimension is still quite mysterious and many generalisations are studied just over the field of complex numbers. A parallel approach to this problem is the 
one based on dgL pairs, used to locally analyse  the cohomology jump loci of some moduli spaces over any field $\kk$ \cite{Budur-Wang, Budur-Rubio, Budur} 

We decide to follow a different approach.
Note that the functor associated to the Brill-Noether loci, or to any deformation problem with cohomological  constraints, is intrinsically more difficult to be studied, since it is not a deformation functor.
Our idea is to get as much information as possible on  infinitesimal deformations of a pair $(\EE, U)$, that we tackle using dgLas even over any algebraically closed field  $\mathbb{K}$ of characteristic zero. 
Then, using the natural link between coherent systems and Brill-Noether theory, we deduce some information on the last deformation problem. In our point of view this approach is very natural and very explicit and takes advantage from the very powerful tool of dgLas in deformation theory.

  In particular, we are able to recover the results 
of  \cite{DE Brill-Noether} about the  infinitesimal deformations of $\EE$ with at least $k$ independent sections (see Section \ref{sezione finale defo con sezioni}) for any algebraically closed field $\kk$ of characteristic zero, adding the hypothesis that $H^i(X, \OO_X)=0$ for $i=1,2$. This additional hypothesis do not allow us  to generalise  the results obtained over  $\CC$, for curves of genus $g>0$.

\smallskip

The paper is organised as follows.   
With the aim of providing an  introduction to the subject, we include   the notion of differential graded Lie algebras, semicosimplicial dgLas and the associated deformation and Deligne functors in Section \ref{sezione preliminare}.

In Section   \ref{sezione deformazione mappe}, we review the notion of infinitesimal deformations of a morphism of locally free sheaves and we describe a dgLa that controls these deformations, as the Thom-Whitney dgLa of a  specific  semicosimplicial dgLa. 

Section \ref{sezione deformazioni (E,U)} is devoted to the infinitesimal deformations of a locally free sheaf together with a subspace of sections, viewed as deformations of a morphism of sheaves. We actually provide a second dgLa which can be constructed explicitly.

Finally, in the last section, we apply our results to study infinitesimal deformations of a locally free sheaf together with at least a prescribed number of independent sections.

\medskip
 
Throughout the paper, we work over an algebraically closed field  $\mathbb{K}$ of characteristic zero;
$\Set$ denotes the category of sets in a fixed universe and $\Art_{\kk}$ the category of local Artinian $\kk$-algebras with residue field $\kk$. For an element $A \in\Art_{\kk} $,  its maximal ideal is indicated by $\m_A$.

\subsection*{Acknowledgements}  Both authors  wish to thank Marco Manetti for useful discussions. 
  We thank the anonymous referee for the careful reading of our manuscript and for his
insightful comments and suggestions, that helped avoiding inaccuracies and improving the quality of the paper.

\section{Preliminaries}\label{sezione preliminare}

\subsection{Differential graded Lie algebras, deformation and Deligne functors}
In this subsection, we introduce the basic definitions and properties of differential graded Lie algebras, together with their associated deformation and Deligne functors. 
 
For full details, we refer the reader to  \cite{Man Pisa, Man Roma, ManettiSeattle, Manetti.LMDT}.
\medskip

\begin{definition}
A \emph{differential graded Lie algebra}, briefly a \emph{dgLa}, is the data $(L,d,[\ ,\ ])$, where $L=\bigoplus_{i\in \mathbb {Z}} L^i$ is a $\mathbb  Z$-graded vector space over $\kk$, $d:L^i \rightarrow L^{i+1}$ is a linear map, such that $d \circ d=0$, and $[\ ,\ ]:L^i \times L^j \rightarrow L^{i+j}$ is a bilinear map, such that:
\begin{itemize}
\item $[\ ,\ ]$ is graded skewsymmetric, i.e., $[a,b]=-(-1)^{\deg a\deg b}[b,a]$,  
\item $[\ ,\ ]$ verifies the graded Jacoby identity, i.e., $[a,[b,c]]=[[a,b],c]+(-1)^{\deg a\deg b}[b,[a,c]]$,
\item $[\ ,\ ]$ and $d$ verify the graded Leibniz's rule, i.e., $d[a,b]=[da,b]+(-1)^{\deg a}[a,db]$, 
\end{itemize}
for every $a, b$ and $c$ homogeneous elements.
\end{definition}

\begin{definition}
Let $(L,d_L, [\ ,\ ]_L)$ and $(M, d_M, [\ ,\ ]_M)$ be two dgLas, a \emph{morphism of dglas} $\varphi:L\to M$ is a degree zero linear morphism that commutes with the brackets and the differentials.  

A \emph{quasi-isomorphism} of dgLas is a morphism of dgLas that induces an isomorphism in cohomology. 
\end{definition}

\begin{definition}\label{defi thom with coppia morfismi}
The \emph{Thom-Whitney homotopy fibre product} of two morphisms of dgLas $h:L \to M$ and $g:N\to M$ is 
the differential graded Lie algebra defined as
\[ L \times_M N :=\{(l,n,m(t,dt)) \in L\times N \times M[t,dt] \mid m(0)=h(l), \  m(1)=g(n)   \}.   \]
 Here, we denote by $M[t,dt]$ the dgLa $M\otimes K[t,dt]$, where $t$ has degree zero, $dt$ has degree $1$ and $d^2 t=0$ as well as $(dt)^2=0$. Moreover, $m(0)$ and $m(1)$ denotes the evaluation in $t=0$ and $t=1$, respectively.
 \end{definition}

\begin{definition}
Let $L$ be a nilpotent differential graded Lie algebra, we define:
$$\Def(L)=\frac{\MC(L)}{\sim_{\textrm{gauge}}},$$
where:
$$\MC(L)=\left\{x\in L^1 \mid dx+\frac{1}{2}[x,x]=0\right\}$$
is the set of the Maurer-Cartan elements,
and the gauge action is the action of $\exp (L^0)$ on $\MC(L)$, given by:
$$e^a * x= x+\sum_{n=0}^{+\infty} \frac{([a,-])^n}{(n+1)!}([a,x]-da).$$

If $L$ is any dgLa, we define the \emph{deformation functor associated to $L$} as the functor \[\Def_L: \Art_{\kk}\to \Set,\] 
that associates to every $A \in \Art_\kk$ the set \[\Def_L(A) := \Def(L\otimes \m_A).\]
\end{definition}  

We recall that the tangent space to the deformation functor $\Def_L$ is the first cohomology space $H^1(L)$ of the dgLa $L$. Moreover, a complete obstruction theory for the functor $\Def_L$ can be naturally defined and its obstruction space is the second cohomology space $H^2(L)$ of the dgLa $L$. 

If the functor of deformations of a geometric object $\mathcal{X}$ is isomorphic to the deformation functor associated to a dgLa $L$, then we say that $L$ \emph{controls} the deformations of $\mathcal{X}$. 

Any morphism $\varphi: L \to M$, induces a morphism $\varphi: \Def_L \to \Def_M$, that is an isomorphism, whenever $\varphi$ is a quasi-isomorphism.

\begin{definition}
A \emph{small category} is a category whose morphisms form a set. 

A \emph{groupoid} is a small category such that every morphism is an isomorphism. We denote the category of groupoids  by $\Grpds$.

For every groupoid $G$, 
the set of isomorphism classes of objects is denoted by  $\pi_0(G)$.

\end{definition}

Let $L$ be a nilpotent dgLa, we define $C(L)$ as the groupoid whose set of objects is $\MC(L)$ and whose morphisms between two objects $x$ and $y$ are defined as the set
\[ \Mor_{C(L)}(x,y) = \{ e^a \in \exp(L^0) \mid e^a * x = y\}. \]
The \emph{irrelevant stabilizer} of a Maurer-Cartan element $x\in\MC(L)$ is the normal subgroup 
\[ I(x) = \{ e^{du + [x,u]} \mid u \in L^{-1} \} \subseteq\Mor_{C(L)}(x,x). \]

\begin{definition}\label{def Deligne groupoid}
The \emph{Deligne groupoid} of a nilpotent differential graded Lie algebra $L$ is the groupoid $\Del(L)$ having as objects the Maurer-Cartan elements of $L$ and as morphisms 
\[ \Mor_{\Del(L)}(x,y) = \frac{\Mor_{C(L)}(x,y)}{I(x)} = \frac{\Mor_{C(L)}(x,y)}{I(y)},   \]
where the second equality is a natural isomorphism (see \cite[Lemma 6.5.5]{Manetti.LMDT}).

If $L$ is any dgLa, we define the \emph{Deligne functor} 
\[\Del_L:\Art_\kk \to \Grpds,\]
as the functor that associates to every $A \in \Art_\kk$ the groupoid 
\[\Del_L(A):= \Del(L\otimes \m_A).\] 
It is immediate from the definitions that there is an equivalence of functors: 
\[ \pi_0(\Del_L)\cong  \Def_L.\]

\end{definition}

\begin{remark}
Note that, if the nilpotent dgLa $L$ is concentrated in non negative degrees, then the irrelevant stabilisers are trivial   and the morphisms of the Deligne groupoid $\Del(L)$ coincide with the ones induced by the gauge action.
\end{remark}

\subsection{Semicosimplical differential graded Lie algebras}
Here, we recall some preliminaries on the  semicosimplicial dgLas, their total object and the deformation functors
 associated to them. We mainly follow \cite{FMM,FIM}, see also  \cite{Manetti.LMDT}.

\begin{definition}
A \emph{semicosimplicial differential graded Lie algebra} is a
covariant functor $\mathbf{\Delta}_{\operatorname{mon}}\to
\mathbf{DGLA}$, from the category
$\mathbf{\Delta}_{\operatorname{mon}}$, whose objects are finite
ordinal sets and whose morphisms are order-preserving injective
maps between them, to the category of dgLas. Equivalently, a
semicosimplicial dgLa ${\mathfrak g}^\Delta$ is a diagram
 \[
\xymatrix{ {{\mathfrak g}_0}
\ar@<2pt>[r]\ar@<-2pt>[r] & { {\mathfrak g}_1}
      \ar@<4pt>[r] \ar[r] \ar@<-4pt>[r] & { {\mathfrak g}_2}
\ar@<6pt>[r] \ar@<2pt>[r] \ar@<-2pt>[r] \ar@<-6pt>[r]&
\cdots},
\]
where each ${\mathfrak g}_i$ is a dgLa, and for each
$i>0$, there are $i+1$ morphisms of dgLas
\[
\partial_{k,i}\colon {\mathfrak g}_{i-1}\to {\mathfrak
g}_{i},
\qquad k=0,\dots,i,
\]
such that
$\partial_{k+1,i+1}\partial_{l,i}=\partial_{l,i+1}\partial_{k,i}$,
for any
$k\geq l$.

A \emph{morphism of semicosimplicial differential graded Lie algebras} $f:{\mathfrak g}^\Delta \to {\mathfrak h}^\Delta$, is given  by a sequence $\{f_i: {\mathfrak g}_i \to {\mathfrak h}_i$\}  of morphisms of dgLas, commuting with the maps $\partial_{k,i}$.
\end{definition}
 
In particular, every $ {\mathfrak
g}_{i}$ is a vector space and so we can consider 
 the graded vector space $\bigoplus_{n\ge 0}\mathfrak{g}_{n}[-n]$ that has two
differentials, i.e.,
\[ d=\sum_{n}(-1)^nd_n,\qquad \text{where}\quad d_n
\text{ is the differential of }\mathfrak{g}_{n},
\]
and
\[
\partial=\sum_{i}\partial_i,\qquad \text{where}
\quad \partial_i=\partial_{0,i}-\partial_{1,i}+\cdots+(-1)^{i}
\partial_{i,i}.
\]

Note that $d\partial+\partial d=0$, thus the graded vector space  $\bigoplus_{n\ge 0}\mathfrak{g}_n[-n]$, endowed
with the differential $D=d+\partial$, is a complex, called the \emph{total complex}, but it can not be endowed with a structure of dgLa.

However, there is a dgLa that is naturally associated to any semicosimplicial dgLa and that is quasi-isomorphic to the total complex. It is constructed as follows. For every $n\ge 0$, we denote by $\Omega_n$ the differential graded
commutative algebra of polynomial differential forms on the
standard $n$-simplex $\Delta^n$:
\[ \Omega_n=\frac{\kk[t_0,\ldots,t_n,dt_0,\ldots,dt_n]}{(\displaystyle \sum_{i=0}^n t_i-1,\displaystyle \sum_{i=0}^n dt_i)}.\]

\begin{definition}
The \emph{Thom-Whitney dgLa} associated to the semicosimplicial dgLa $\g^{\Delta}$ is
\[\operatorname{Tot}(\mathfrak{g}^\Delta) =\{ (x_n)_{n}\in \prod_n \Omega_n\otimes {\mathfrak g}_n
\mid \delta^{k,n}x_n= \partial_{k,n}x_{n-1}\quad \forall\; 0\le k\le n\},\]
where, for $k=0,\ldots,n$, $\delta^{k,n}\colon \Omega_n\to \Omega_{n-1}$ are the face maps and $\partial_{k,n}\colon \g_{n-1}\to \g_n$ are the maps of the semicosimplicial dgLa $\g^{\Delta}.$

 As already mentioned, $\Tot(\g^\Delta)$ is quasi-isomorphic, as  graded vector space, to the total complex $\left(\bigoplus_{n\ge 0}\mathfrak{g}_n[-n], d + \partial\right)$. 
\end{definition}

 \begin{remark}
 Let $h:L \to M$ and $g:N\to M$ be two morphisms of dgLas and consider the semicosimplicial dgLa
  \[
 \xymatrix{ L\oplus N
\ar@<2pt>[r]^h\ar@<-2pt>[r]_g & M
      \ar@<4pt>[r] \ar[r] \ar@<-4pt>[r] &0
\ar@<6pt>[r] \ar@<2pt>[r] \ar@<-2pt>[r] \ar@<-6pt>[r]&
\cdots}.
\]
 In this case, the associated Thom-Whitney dgLa  is nothing else that the 
 Thom-Whitney homotopy fibre product of Definition \ref{defi thom with coppia morfismi}.
 \end{remark}

\begin{example}
Given a sheaf ${\mathcal L}$ of dgLas on a topological space $X$ and an open cover $\mathcal{V}=\{V_i\}_i$ of $X$, the \v{C}ech semicosimplicial  dgLa ${\mathcal L}({\mathcal V})$ is given by:
\[
\xymatrix{ {\prod_i\mathcal{L}(V_i)}
\ar@<2pt>[r]\ar@<-2pt>[r] & {
\prod_{i<j}\mathcal{L}(V_{ij})}
      \ar@<4pt>[r] \ar[r] \ar@<-4pt>[r] &
      {\prod_{i<j<k}\mathcal{L}(V_{ijk})}
\ar@<6pt>[r] \ar@<2pt>[r] \ar@<-2pt>[r] \ar@<-6pt>[r]& \cdots},
\]
where, as usual, $V_{ij}=V_i \cap V_j$, $V_{ijk}=V_i \cap V_j \cap V_k$ and so on, denote the intersections, $\mathcal{L}(V_i)=\Gamma(V_i, \mathcal{L})$ stays for the sections and  the morphisms $\partial_{k,i}$ are the restriction maps.
Here, the  total complex associated to
 the  \v{C}ech semicosimplicial dgLie algebra  
$\mathcal{L}(\mathcal{V})$
is  the \v{C}ech
complex $\check{C}(\mathcal{V},\mathcal{L})$ of the sheaf $\mathcal{L}$. In particular, by the quasi isomorphism between the total complex and the Thom-Whitney dgLa, we have  $H^k(\Tot(\mathcal{L}(\mathcal{V}))) \cong \check{H}^k(\mathcal{V}, \mathcal{L})$, for all $k \in \ZZ$.

\end{example}

\bigskip

According to \cite{FMM,FIM}, we define the following deformation functor associated to a semicosimplicial dgLa.

\begin{definition} \label{Def Zsc}
Let $\mathfrak g^{\Delta}$ be a semicosimplicial dgLa. The functor
\[Z^1_{\rm sc}(\exp \g^{\Delta})    :
\mathbf{Art}_{\mathbb K} \to \mathbf{Set}\] is defined, for all
$A\in \mathbf{Art}_{\mathbb K}$, by
\[ Z^1_{\rm sc}(\exp \g^{\Delta})(A)= \left\{ (l,m)\in
(\g_0^1 \oplus \g^0_1) \otimes \mathfrak m_A \left|
\begin{array}{l}
dl+\frac{1}{2}[l,l]=0,\\ \partial_{1,1}l=e^{m}*\partial_{0,1}l, \\
{\partial_{0,2}m} \bullet {-\partial_{1,2}m} \bullet
{\partial_{2,2}m} =dn+[\partial_{2,2}\partial_{0,1}l,n]\\
\qquad\qquad\qquad\qquad \text{for some $n\in {\mathfrak g}_2^{-1}\otimes{\mathfrak m}_A$}
\end{array} \right.\right\},\]
where the symbol $ \bullet$ stands for the Baker Campbell Hausdorff product in a Lie algebra.
\end{definition}

Two elements
$(l_0,m_0)$ and $(l_1,m_1) \in Z^1_{\rm sc} (\exp \g^{\Delta})(A)$
are equivalent under the relation  $\sim$ if  and only if there
exist elements $a \in \g^0_0\otimes \m_A$ and $b\in{\mathfrak
g}_1^{-1}\otimes{\mathfrak m}_A$ such that
\begin{equation} \label{def rel equiv}
\begin{cases}
e^a * l_0=l_1\\
- m_0\bullet -\partial_{1,1}a \bullet m_1
\bullet \partial_{0,1}a=db+[\partial_{0,1}l_0,b].
\end{cases}
\end{equation}

\begin{definition} \label{def.h1sc}
Let $\g^\Delta$ be a semicosimplicial dgLa, the functor
\[ H^1_{\rm sc}(\exp \g^\Delta)  : \mathbf{Art}_{\mathbb K}
 \to \mathbf{Set}\]
is defined, for all $A\in \mathbf{Art}_{\mathbb K}$, by
\[H^1_{\rm sc}(\exp \g^\Delta)(A)= \frac{Z^1_{\rm sc}
(\exp \g^\Delta)(A)}{\sim}.\]
\end{definition}

 Note that any morphism of semicosimplicial dgLas $\g^{\Delta} \to \h^{\Delta}$ induces a natural transformation of functors $H^1_{\rm sc}(\exp \g^\Delta) \to H^1_{\rm sc}(\exp \h^\Delta)$.

\begin{remark}
If $\g^\Delta$ is a semicosimplicial  Lie
algebra, i.e., every dgLa ${\mathfrak g}_i$ is concentrated in degree zero, then the functor
$H^1_{\rm sc}(\exp \g^\Delta)$ reduces to
the one defined in \cite{FMM}.

Moreover, if $\g^\Delta$ is a semicosimplicial dgLa, such that every dgLa is concentrated in non negative degrees, 
we can easily view the above functor as a functor in groupoids. More explicitly, we have
\[ \widetilde{H}^1_{\rm sc}(\exp \g^\Delta): \Art_\kk \to \Grpds, \]
where, for any $A\in \Art_\kk$, the set of object is  
 \begin{equation} \label{def Z1 caso gradi non neg}
Z^1_{\rm sc}(\exp \g^{\Delta})(A)= \left\{ (l,m)\in
(\g_0^1 \oplus \g^0_1) \otimes \mathfrak m_A \left|
\begin{array}{l}
dl+\frac{1}{2}[l,l]=0,\\ \partial_{1,1}l=e^{m}*\partial_{0,1}l, \\
{\partial_{0,2}m} \bullet {-\partial_{1,2}m} \bullet
{\partial_{2,2}m} =0\\
\end{array} \right.\right\},\end{equation}
and the isomorphisms between two objects 
$(l_0,m_0)$ and $(l_1,m_1) \in Z^1_{\rm sc} (\exp \g^{\Delta})(A)$
are given by the elements $a \in \g^0_0\otimes \m_A$ as above, i.e., such that
\begin{equation} \label{def rel equiv caso gradi non neg}
\begin{cases}
e^a * l_0=l_1\\
- m_0\bullet -\partial_{1,1}a \bullet m_1
\bullet \partial_{0,1}a=0.
\end{cases}\end{equation}
\end{remark}

 We recall the following result \cite[Theorem 4.10]{FIM}, that relates these functors with the ones associated to the dgLas.
\begin{theorem}   \label{teo.H1sc def functors}
Let $\g^\Delta$ be a semicosimplicial dgLa, such that $H^j(\g^i)=0$ for all $i\geq 0$ and $j <0$. Then, there is an equivalence of functors:
\[
\Def_{\Tot(\g^{\Delta})} \cong H^1_{\rm sc}(\exp \g^\Delta).
\]
In particular, 
  the functor $H^1_{\rm sc}(\exp \g^\Delta)$ is a deformation functor, its tangent space is $H^1(\Tot(\g^{\Delta}))$ and its obstructions are contained in $H^2(\Tot(\g^{\Delta}))$.
\end{theorem}

In particular, if  the functor of deformations of a  geometric object $\mathcal{X}$ is isomorphic to the deformation 
functor $H^1_{\rm sc}(\exp \g^\Delta)$ associated to a  semicosimiplicial dgLa $\g^\Delta$, then  the Thom-Whitney dgLa ${\Tot(\g^{\Delta}})$  controls the deformations of $\mathcal{X}$.

\begin{example} \label{ex. def E via semicosimpl}
Let $\EE$ be a locally free sheaf of $\OO_X$-modules on a projective variety $X$. Denote by $\EE nd (\EE)$ the sheaf of  $\OO_X$-modules endomorphisms of $\EE$. It is a classical fact that $\EE nd (\EE)$ encodes all the information of the infinitesimal deformations of $\EE$. Over the field of  complex numbers $\CC$, it can be rephrased saying that the dgLa $A_X^{0,*}(\EE nd(\EE))$ controls the deformations of $\EE$ via the Maurer-Cartan functor modulo the gauge equivalence \cite{Fuk}. Over  any algebraically closed field  $\mathbb{K}$ of characteristic zero,  one can consider the \v{C}ech semicosimplicial Lie algebra associated to the sheaf $\EE nd (\EE)$ and with an open affine cover $\mathcal{V}=\{V_i\}_i$ of $X$: 
\[ \EE nd (\EE)(\mathcal V): 
\xymatrix{ {\prod_i\EE nd (\EE)(V_i)}
\ar@<2pt>[r]\ar@<-2pt>[r] & {
\prod_{i<j}\EE nd (\EE) (V_{ij})}
      \ar@<4pt>[r] \ar[r] \ar@<-4pt>[r] &
      {\prod_{i<j<k}\EE nd (\EE)(V_{ijk})}
\ar@<6pt>[r] \ar@<2pt>[r] \ar@<-2pt>[r] \ar@<-6pt>[r]& \cdots}.
\]
In \cite{FMM} the authors proved that the functor $H_{sc}^1(\exp \EE nd (\EE)(\mathcal V))$ of Definition \ref{def.h1sc} is isomorphic to the functor of  infinitesimal deformations of $\EE$. 
 Moreover, the cohomology of the total complex of the semicosimplicial dgLa above is $H^k(\Tot(\EE nd (\EE)(\mathcal V))) \cong H^k(X, \EE nd (\EE))$ for all $k \in \ZZ$, according to the classical fact  that the tangent space to the functor of  the infinitesimal deformations of $\EE$ is $H^1(X, \EE nd (\EE))$ and the obstructions are contained in $H^2(X, \EE nd (\EE))$.

\end{example}

\subsection{Semicosimplicial groupoid and descent theorem}

This subsection is dedicated to  semicosimplicial objects in the category of groupoids, their total groupoid and the fundamental Hinich's Theorem on descent of Deligne groupoids. Here, we mainly follow \cite{Manetti.LMDT}. 
\begin{definition}
A \emph{semicosmplicial groupoid} is a
covariant functor $\mathbf{\Delta}_{\operatorname{mon}}\to
\Grpds$, from the category
$\mathbf{\Delta}_{\operatorname{mon}}$, whose objects are finite
ordinal sets and whose morphisms are order-preserving injective
maps between them, to the category of groupoids. Equivalently, a
semicosimplicial groupoid $G^\Delta$ is a diagram
 \[
\xymatrix{ G_0
\ar@<2pt>[r]\ar@<-2pt>[r] & G_1
      \ar@<4pt>[r] \ar[r] \ar@<-4pt>[r] & G_2
\ar@<6pt>[r] \ar@<2pt>[r] \ar@<-2pt>[r] \ar@<-6pt>[r]&
\cdots},
\]
where each ${G}_i$ is a groupoid, and for each
$i>0$, there are $i+1$ morphisms   of  groupoids
\[
\partial_{k,i}\colon {G}_{i-1}\to {G}_{i},
\qquad k=0,\dots,i,
\]
such that
$\partial_{k+1,i+1}\partial_{l,i}=\partial_{l,i+1}\partial_{k,i}$,
for any
$k\geq l$. Since every groupoid is a small category, equality of morphisms of groupoids is intended in the strict sense.

A \emph{morphism of semicosimplicial groupoids} $f:G^\Delta \to H^\Delta$, is given  by a sequence $\{f_i: G_i \to H_i$\}  of morphisms of groupoids, commuting with the maps $\partial_{k,i}$.
\end{definition}

\begin{example}
Let  ${\mathfrak g}^\Delta$ be a semicosimplicial dgLa
 \[
\xymatrix{ {{\mathfrak g}_0}
\ar@<2pt>[r]\ar@<-2pt>[r] & { {\mathfrak g}_1}
      \ar@<4pt>[r] \ar[r] \ar@<-4pt>[r] & { {\mathfrak g}_2}
\ar@<6pt>[r] \ar@<2pt>[r] \ar@<-2pt>[r] \ar@<-6pt>[r]&
\cdots},
\]
such that every $ {\mathfrak g}_i$ is a nilpotent dgLa. Applying the Deligne groupoid of Definition \ref{def Deligne groupoid}, one obtains the semicosimplicial groupoid 
\[
\Del_{\g^{\Delta}}: \xymatrix{  \Del_{\g_0}
\ar@<2pt>[r]\ar@<-2pt>[r] &  \Del_{\g_1}
      \ar@<4pt>[r] \ar[r] \ar@<-4pt>[r] &  \Del_{\g_2}
\ar@<6pt>[r] \ar@<2pt>[r] \ar@<-2pt>[r] \ar@<-6pt>[r]&
\cdots}.
\]
Analogously, for any $A \in \Art_\kk$ and any  semicosimplicial dgLa ${\mathfrak g}^\Delta$, one defines the semicosimplicial groupoid
\[
\Del_{\g^{\Delta}}(A): \xymatrix{  \Del_{\g_0}(A)
\ar@<2pt>[r]\ar@<-2pt>[r] &  \Del_{\g_1}(A)
      \ar@<4pt>[r] \ar[r] \ar@<-4pt>[r] &  \Del_{\g_2}(A)
\ar@<6pt>[r] \ar@<2pt>[r] \ar@<-2pt>[r] \ar@<-6pt>[r]&
\cdots}.
\]

\end{example}

\begin{definition}
Let $G^\Delta$ be a semicosimplicial groupoid. The \emph{total groupoid} $\Tot(G^\Delta)$ of  $G^\Delta$ is the groupoid defined as follows:
\begin{itemize}
\item the objects of $\Tot(G^\Delta)$ are the elements of the form $(l,m)$, where $l$ is an object in $G_0$ and $m: \partial_{0,1} l \to \partial_{1,1} l$ is a morphism in $G_1$ such that the following diagram commutes  in $G_2$

 \[ \xymatrix{ &     \partial_{0,2}\partial_{0,1} l  \ar@{=}[dl]   \ar[dr]^{\partial_{0,1} m}     &  \\
                 \partial_{1,2}\partial_{0,1} l  \ar[d]^{\partial_{1,1} m}  &         &  \partial_{0,2}\partial_{1,1} l \ar@{=}[d]    \\
 \partial_{1,2}\partial_{1,1} l\ar@{=}[dr]  &  &\partial_{2,2}\partial_{0,1} l \ar[dl]^{\partial_{2,1} m}  \\
 &  \partial_{1,2}\partial_{1,1} l,  &   }
   \]

\item the morphisms between two objects $(l_0,m_0)$ and $(l_1,m_1)$ are the morphisms $a$ in $G_0$ such that the following diagram commutes
\[ \xymatrix{ \partial_{0,1} l_0    \ar[r]^{m_0}\ar[d]_{\partial_{0,1}a} & \partial_{1,1}  l_0 \ar[d]^{\partial_{1,1}a} \\  
\partial_{0,1} l_1   \ar[r]_{m_1} &\partial_{1,1}  l_1.  }   \] 
\end{itemize}
\end{definition}

\begin{proposition}\cite[Proposition 7.5.5]{Manetti.LMDT}\label{proposizione equivalenza gruppp semicosim e Tot}
Let $f: G^\Delta \to H^\Delta$ be a morphism of semicosimplicial groupoids, such that for every $n=0,1,2$ the component $f_n: G_n \to H_n$ is an equivalence
of groupoids, then $f: \Tot(G^\Delta) \to \Tot(H^\Delta)$ is also an equivalence of groupoids.
\end{proposition}

We end up this section with the following theorem due to Hinich on descent of Deligne groupoids.
\begin{theorem} \cite[Corollary 4.1.1]{hinich} \label{hinich theorem}
Let $\g^{\Delta}$ be a nilpotent semicosimplicial dgLa, such that every $\g_n$ is trivial in negative degrees.  Then, there exists a natural equivalence of groupoids
\[\Del_{\Tot(\g^{\Delta})} \to \Tot(\Del_{\g^{\Delta}}).
\]
\end{theorem}

\begin{remark}
We remark that  a generalisation of  the previous Hinich's descent Theorem for the Deligne-Getzler $\infty$-groupoid was proved in \cite{bandiera}.
\end{remark}

\section{Deformations of a morphism of locally free sheaves} \label{sezione deformazione mappe}
In this section, we  would like to analyse the infinitesimal deformations of a morphism of locally free sheaves and explicitly describe a differential graded Lie algebra that controls the deformation problem.

\subsection{Geometric deformations}
Let $X$ be a smooth projective variety over $\mathbb{K}$,   $\FF$ and $\GG$  locally free sheaves of $\OO_X$-modules over $X$ and   $\alpha \colon \FF\to \GG$   a morphism of them. 
First of all we recall some classical definitions.  

\begin{definition} \emph{An infinitesimal deformation of $\FF$ over $A\in \Art_{\kk}$} is a locally free sheaf  $\FF_A$ of $\OO_X \otimes A$-module  on $X\times \Spec A $ together with a morphism $\pi_A: \FF_A \to \FF$, such that the obvious restriction of scalars $\pi_A: \FF_A \otimes_A \kk \to \FF$ is an isomorphism. 

Let $( \FF_A, \pi_A)$ and  $( {\FF'}_A, {\pi'}_A)$ be two deformations of the sheaf $\FF$ over $A$. They are \emph{isomorphic}, if there exists an isomorphism $\phi \colon  \FF_A \to {\FF'}_A$ of $\OO_X \otimes A$-modules, that commutes with the restrictions to $\FF$. 
\end{definition}
\begin{definition} \label{def.deform di morf di fasci}
 \emph{An infinitesimal deformation of the  morphism  $\alpha \colon \FF\to \GG$ over  $A\in \Art_{\mathbb{K}}$} is a morphism $\alpha_A: \FF_A\to \GG_A$ of locally free sheaves of $\OO_X \otimes A$-modules over $X\times \Spec A$, where $\FF_A$ and $\GG_A$ are deformations of $\FF$ and $\GG$ over $A$ respectively, such that the following diagram is commutative: 
\[
\xymatrix{ \FF_A \ar[r]^{\alpha_A} \ar[d]^{\pi_A^\FF}  & \GG_A \ar[r] \ar[d]^{\pi_A^\GG} & \Spec A  \ar[d] \\
\FF \ar[r]^{\alpha} & \GG \ar[r] & \Spec \kk .  }
\]

Two deformations $\alpha_A: \FF_A\to \GG_A$ and ${\alpha'}_A: {\FF'}_A\to {\GG'}_A$  of $\alpha: \FF \to \GG$ over $A$ are \emph{isomorphic}, 
if there exist a pair $(\phi,\psi)$ of isomorphisms of sheaves $\phi: \FF_A \to {\FF'}_A$ and $\psi : \GG_A \to {\GG'}_A$, such that the following diagram commutes:
\[
\xymatrix{ \FF_A \ar[rrr]^{{\alpha}_A} \ar[dd]^\phi \ar[dr]^{\pi_A^{\FF}}  & & & \GG_A \ar[dl]_{\pi_A^{\GG}}   \ar[dd]^\psi \ar[dr] & \\
 & \FF \ar[r]^\alpha & \GG  & & \Spec A \\
{\FF'}_A \ar[rrr]_{{\alpha'}_A } \ar[ur]_{\pi_A^{\FF'}}   & & &     {\GG'}_A \ar[ur] \ar[ul]^{\pi_A^{\GG'}} . &  }\]

\end{definition}

We recall that, the trivial deformation of a sheaf $\FF$ over $A$ is given by  $\FF\otimes_{\OO_X}\OO_ {X \times \Spec A}=\FF\otimes_\kk A $ and that the trivial deformation of $\alpha \colon \FF\to \GG$   over  $A\in \Art_\kk$ is given by the trivial extension 
$ \alpha\otimes \Id_A \colon \FF\otimes_\kk A  \to \GG\otimes_\kk A $.
 
Note that, since $\FF$ and $\GG$ are  locally free sheaves, any  infinitesimal deformation  of $\alpha \colon \FF\to \GG$  is locally trivial, i.e.,  it is locally in $X$  isomorphic to the trivial  deformation.

\begin{definition} \label{def.funtore def in gruppoidi di alpha}
The \emph{functor of infinitesimal deformations of the morphism $\alpha:\FF \to \GG$} is the  functor
\[ \Def_{(\FF,\alpha, \GG)} : \Art_\kk \to \Grpds,\]
that associates, to any $A \in \Art_\kk$, the groupoid $\Def_{(\FF,\alpha, \GG)}(A)$, whose objects are the deformations of the morphism $\alpha$  over
 $A$ and whose morphisms are the isomorphisms of them.
  Note that  
\[ \pi_0(\Def_{(\FF,\alpha, \GG))(A)}) = \{\mbox{isomorphism classes of deformations of the morphism $\alpha$ over $A$}\} \]
is the classical functor of deformations in $\Set$.

\end{definition} 

\begin{remark}
Among all the  infinitesimal deformations of $\alpha \colon \FF\to \GG$ over $A \in\Art_\kk$, there are the infinitesimal deformations of the morphism  $\alpha$ in which $\FF$ and $\GG$ deform trivially, i.e., the deformations $ \alpha_A \colon \FF\otimes_\kk A  \to \GG\otimes_\kk A $ such that just the map deforms. The  groupoid of these deformations defines a subfunctor 
$\Def_{\alpha} : \Art_\kk \to \Grpds$ of the functor $\Def_{(\FF,\alpha, \GG)} $.
\end{remark}

Let  $\alpha \colon \FF\to \GG$  be a morphism of locally free sheaves and $\gamma$ its graph in
$ \FF\oplus \GG$. We define $\gamma$ as the image of the morphism of sheaves $(\Id, \alpha): \FF \to \FF\oplus \GG$. 

For future use, we would like to observe  that a deformation of $\alpha: \FF \to \GG$ over $A \in \Art_\kk$    as  in Definition \ref{def.deform di morf di fasci} can be also seen as the collection of the following data:

\begin{itemize}
\item two deformations $\FF_A$ and $\GG_A$ of $\FF$ and $\GG$ over $A$, respectively; 
\item a deformation $\gamma_A \subseteq (\FF\oplus\GG)_A$ of the   inclusion  $\gamma \subseteq \FF\oplus\GG$ over $A$; 
\item an isomorphism $f_A$ between the deformations  $\FF_A \oplus \GG_A$ and $(\FF\oplus\GG)_A$. 
\end{itemize}
Two of such collections of data $\left(\FF_A, \GG_A, \gamma_A \subseteq (\FF\oplus\GG)_A, f_A \right)$ and $\left(\FF'_A, \GG'_A, \gamma'_A \subseteq (\FF\oplus\GG)'_A, f'_A \right)$ define isomorphic deformations of $\alpha:\FF \to \GG$, if there exist:
\begin{itemize}
\item two isomorphisms $\phi: \FF_A\to \FF'_A$ and $\psi:\GG_A\to \GG'_A$ of the deformations of $\FF$ and $\GG$ over $A$, respectively, 
\item  
an isomorphism $\chi:(\FF\oplus\GG)_A\to (\FF\oplus\GG)'_A$ of the deformations of $\FF\oplus \GG$ over $A$, with $\chi(\gamma_A)  \subseteq \gamma'_A$,
\end{itemize}
such that the following diagram is commutative
\[ \xymatrix{ \FF_A \oplus \GG_A \ar[r]^{f_A} \ar[d]^{(\phi,\psi)}  & (\FF\oplus\GG)_A \ar[d]^\chi \\
\FF'_A \oplus \GG'_A \ar[r]^{f'_A}  & (\FF\oplus\GG)'_A .
  }  \]

We point out that this approach is similar to the one used in \cite{HorikawaI, HorikawaII, HorikawaIII, Sernesi} to analyse deformations of holomorphic maps of complex manifolds. In particular, this was applied in \cite{Dona.Tesi, Dona.def maps} and in \cite{Manetti.LMDT} to investigate these deformations via dgLas.

\bigskip

\subsection{ The local case} \label{subsection.the local case}
First of all, we analyse the infinitesimal deformations of a morphism $\alpha: \FF \to \GG$  in the local case.
We assume that $X$ is a smooth affine variety over $\kk$ and  $\alpha: \FF \to \GG$ is a morphism of the  free sheaves $\FF$ and $\GG$ on $X$. 
Under this hypothesis, we are able to find out a dgLa that controls the deformations of $\alpha$.
 
We denote  by  $\EE nd (\FF)$, $\EE nd (\GG)$  and  $\EE nd (\FF \oplus \GG)$ the free sheaves  of the $\OO_X$-modules  endomorphisms of the sheaves $\FF$, $\GG$ and $\FF \oplus \GG$, respectively. Moreover, let $\LL$ be the subsheaf of $\EE nd (\FF \oplus \GG)$  that preserves the graph  $\gamma \subseteq \FF \oplus \GG$ of the morphism $\alpha$, i.e., 
\[ \LL := \{ \varphi \in \EE nd (\FF \oplus \GG)  \mid \varphi(\gamma) \subseteq \gamma \}.
\]
 
 \begin{definition} \label{def.H(F,alpha,G)}
In the above notation, we define the dgLa $H(X, \FF, \alpha, \GG)$
as the Thom-Whitney homotopy fibre product of the diagram of Lie algebras
\[ \xymatrix{  & \Gamma(X, \LL) \ar[d]^h \\
\Gamma(X, \EE nd(\FF) \oplus \EE nd (\GG) ) \ar[r]^-g  &   \Gamma(X, \EE nd (\FF \oplus \GG)),} \]
where $h$ is the inclusion and $g$ is defined as the map 
$(\varphi, \psi) \mapsto\left( \begin{array}{cc} \varphi & 0 \\ 0 & \psi\end{array}\right).
$
More explicitly, according to Definition \ref{defi thom with coppia morfismi}, the elements of  $H(X, \FF, \alpha, \GG)$ are of the form $(x, y, z(t,dt))$, where
\[ x \in \Gamma(X, \LL),  \quad y \in  \Gamma(X, \EE nd(\FF) \oplus \EE nd (\GG)), \]\[z(t, dt)  \in  \Gamma(X, \EE nd (\FF \oplus \GG))[t, dt], \]
such that $ z(0)= h(x), z(1) = g(y)$.
 
\end{definition}

Then, we can prove the following result.
 
\begin{theorem} \label{thm.equvalenza di funtori caso Stein}
In this setting, the functor $\Def_{(\FF, \alpha, \GG)}$ of   infinitesimal deformations of the morphism $\alpha: \FF \to \GG$ on $X$ is equivalent to the Deligne functor
$\Del_{H(X,\FF, \alpha, \GG)}$ associated to the dgLa $H(X, \FF, \alpha, \GG)$.
\end{theorem}

\begin{proof}
By Definition \ref{def.H(F,alpha,G)}, $H(X,\FF, \alpha, \GG)$ is the Thom-Whitney dgLa of the following semicosimplicial Lie algebra: 
\[ \h^\Delta:
\xymatrix{ {\Gamma(X, \LL) \oplus \Gamma(X, \EE nd(\FF) \oplus \EE nd (\GG)) }
\ar@<2pt>[r]^{\ \ \ \ \ \ \ \ \ \ h}\ar@<-2pt>[r]_{\ \ \ \ \ \ \ \ \ \ g} & {\Gamma(X, \EE nd (\FF \oplus \GG))}
      \ar@<4pt>[r] \ar[r] \ar@<-4pt>[r] & 0.}
\]
By Hinich's Theorem on descent of Deligne groupoids (Theorem \ref{hinich theorem}), there is an equivalence of functors of groupoids 
\[\Del_{H(X,\FF, \alpha, \GG)}  \cong \Tot(\Del_{\h^\Delta}),   \]
i.e., for every $A \in \Art_\kk$, there is an equivalence of the groupoids  
$\Del_{H(X,\FF, \alpha, \GG)} (A)  \cong \Tot(\Del_{\h^\Delta})(A).$ 

Let us describe explicitly the objects and the morphisms of the groupoid $\Tot(\Del_{\h^\Delta})(A)$, for any $A \in \Art_\kk$.
Its objects are elements of the form $(x,y,z)$, where:
\[ x \in \Del(\Gamma(X, \LL) \otimes \m_A), \quad y \in \Del(\Gamma(\EE nd (\FF) \oplus \EE nd (\GG)) \otimes \m_A), \]
and so $x=y=0$, since we are dealing with Lie algebras, and
\[  z = e^w \in \exp (\Gamma(X, \EE nd (\FF \oplus \GG)) \otimes \m_A), \]
such that $e^w * 0 = 0 \in \Gamma(X, \LL)\otimes \m_A$. 
A morphism between two of such objects, $e^w$ and $e^t$, is 
of the form $e^a$, where $a=(a_1,a_2,a_3) \in (\Gamma(X,\EE nd (\FF)  \oplus \EE nd (\GG)) \oplus \Gamma(X,\LL)) \otimes \m_A$, such that $e^w e^{h(a)} = e^t e^{g(a)}.$

These data corresponds to the data of the objects and morphisms of the groupoid $\Def_{(\FF,\alpha,\GG)}(A)$. Indeed, when $X$ is an affine smooth variety, any deformation  of a free sheaf on $X$ is trivial and the only datum of a deformation of $\alpha$ is an isomorphism of the direct sum of the trivial deformations of $\FF$ and $\GG$ to the trivial deformation of $\FF \oplus \GG$, that preserves $\gamma$, this is $e^w$.

 A morphism between two of such deformations, $e^w$ and $e^t$, is given by an isomorphism of the trivial deformation of $\FF$, an isomorphism of the trivial deformation of $\GG$, an isomorphism of the trivial deformation of $\FF \oplus \GG$ that preserve $\gamma$. These data are $(a_1,a_2,a_3)$, respectively, and they have to respect compatibilities with the isomorphisms above, that are expressed by the equation $e^w e^{h(a)} = e^t e^{g(a)}$.
\end{proof}

\subsection{The global case} \label{subsection.the global case}
 
Let, now, $X$ be a smooth projective variety over $\kk$,   $\alpha: \FF \to \GG$    a morphism of locally free sheaves on $X$ and $\VV = \{ V_i \}_i$ an open affine cover of $X$ trivialising both $\FF$ and $\GG$.

Consider the following semicosimplicial dgLa
\[ H(\VV)^\Delta:
\xymatrix{ {\prod_i H(V_i,\FF, \alpha, \GG)}
\ar@<2pt>[r]\ar@<-2pt>[r] & {
\prod_{i<j}H(V_{ij},\FF, \alpha, \GG)}
      \ar@<4pt>[r] \ar[r] \ar@<-4pt>[r] &
      {\prod_{i<j<k}H(V_{ijk},\FF, \alpha, \GG)}
\ar@<6pt>[r] \ar@<2pt>[r] \ar@<-2pt>[r] \ar@<-6pt>[r]& \cdots},
\]
where $H(V_i,\FF, \alpha, \GG)$ is the Thom-Whitney homotopy fibre  of  Definition \ref{def.H(F,alpha,G)}. 

Then, we can prove the following result.

\begin{theorem} \label{thm.Del eq Def}
 In this setting, the functor $\Def_{(\FF, \alpha, \GG)}$ of  infinitesimal deformations of the morphism $\alpha: \FF \to \GG$ on $X$ is equivalent to the Deligne functor
$\Del_{\Tot(H(\VV)^\Delta)}$ associated to the Thom-Whitney dgLa of the semicosimplicial dgLa $H(\VV)^\Delta$.

\end{theorem}
 
\begin{proof}
Applying Hinich's Theorem  on descent of Deligne groupoids (Theorem \ref{hinich theorem}), there exists an equivalence of groupoids:
\[ \Del_{\Tot(H(\VV)^\Delta)} \cong \Tot(\Del_{H(\VV)^\Delta}).\]
By Proposition \ref{proposizione equivalenza gruppp semicosim e Tot} and by the local case, analysed in Theorem \ref{thm.equvalenza di funtori caso Stein}, we have 
\[ \Tot(\Del_{H(\VV)^\Delta}) \cong \Tot(\Def_{(\FF, \alpha, \GG)(\VV)^\Delta}).   \]
Here $\Def_{(\FF, \alpha, \GG)(\VV)^\Delta}$ is the semicosimplicial functor in groupoids 
\[
\xymatrix{ {\prod_i \Def_{(\FF, \alpha, \GG)(V_i)}}
\ar@<2pt>[r]\ar@<-2pt>[r] & {
\prod_{i<j}\Def_{(\FF, \alpha, \GG)(V_{ij})}}
      \ar@<4pt>[r] \ar[r] \ar@<-4pt>[r] &
      {\prod_{i<j<k}\Def_{(\FF, \alpha, \GG)(V_{ijk})}}
\ar@<6pt>[r] \ar@<2pt>[r] \ar@<-2pt>[r] \ar@<-6pt>[r]& \cdots},
\]
where $\Def_{(\FF, \alpha, \GG)(V)}$ is the functor of infinitesimal deformations of $\alpha_{|V}: \FF_{|V} \to \GG_{|V}$,  for any $V\subset X$ affine open subset that trivialises both $\FF$ and $\GG$.
Moreover, by a classical argument, global deformations are given by gluing local 
ones, then 
\[ \Tot(\Def_{(\FF, \alpha, \GG)(\VV)^\Delta}) \cong \Def_{(\FF, \alpha, \GG)},\]
as desired. 
\end{proof}
\begin{corollary}
In the above notation, the Thom-Whitney dgLa associated to the semicosimplicial dgLa $H(\VV)^\Delta$ controls the   infinitesimal deformations of the morphism of   locally free  sheaves $\alpha:\FF \to \GG$. 
\end{corollary}
 \begin{remark}
We note that, for any coherent sheaf of dgLas $H$, the quasi isomorphism class of  the Thom-Whitney dgLa $ \Tot(H(\VV)^\Delta)$ does not depend on the choice of the affine open cover $\VV$ of $X$ \cite{FIM}. Then, the previous construction does not depend on the choice of the cover.
 \end{remark}

\begin{remark}\label{remark thom whit non dipende da verso}
 Let $\alpha: \FF \to \GG$ be a morphism of locally free sheaves of $\OO_X$-modules  on $X$ and let $\VV=\{V_i\}_i$ be an affine open cover of $X$, we can construct a bisemicosimplicial object associated to these data as follows. On every open set $V_i \in \VV$, we have
  \[  
\xymatrix{ {\Gamma(V_i, \LL) \oplus \Gamma(V_i, \EE nd(\FF) \oplus \EE nd (\GG)) }
\ar@<2pt>[r]^{\ \ \ \ \ \ \ \ \ \ h}\ar@<-2pt>[r]_{\ \ \ \ \ \ \ \ \ \ g} & {\Gamma(V_i, \EE nd (\FF \oplus \GG))}
      \ar@<4pt>[r] \ar[r] \ar@<-4pt>[r] & 0.}
\]
Since $h$ and $g$ are morphisms of sheaves, they commute with restrictions of every open subsets, inducing morphisms of the \v{C}ech semicosimplicial objects
  \[  
\xymatrix{ {  \LL(\VV) \oplus   \EE nd(\FF)(\VV) \oplus \EE nd (\GG)(\VV) }
\ar@<2pt>[r]^{\ \ \ \ \ \ \ \ \ \ h}\ar@<-2pt>[r]_{\ \ \ \ \ \ \ \ \ \ g} &   \EE nd (\FF \oplus \GG)(\VV)
      \ar@<4pt>[r] \ar[r] \ar@<-4pt>[r] & 0.}
\] 
In Theorem \ref{thm.Del eq Def}, we define, for every open set $V_i$, the dgLas $H(V_i,\FF, \alpha, \GG)$ as the Thom-Withney homotopy fiber products; they form a semicosimplicial dgLas $H(\VV)^\Delta$ and we consider the associated Thom-Withney dgLa  $\Tot(H(\VV)^\Delta)$. This construction does not depend on the order \cite{dona.rendiconti}. We could first consider the Thom-Withney dgLas of the \v{C}ech semicosimplicial sheaves of Lie algebras $\LL(\VV)$,  $\EE nd(\FF)(\VV)$,  $\EE nd (\GG)(\VV)$ and $\EE nd (\FF \oplus \GG)(\VV)$ and obtain the semicosimplicial dgLa
\begin{equation}\label{successione nella remark}
\xymatrix{ { \Tot( \LL(\VV) )\oplus   \Tot( \EE nd(\FF)(\VV)) \oplus \Tot( \EE nd (\GG)(\VV)) }
\ar@<2pt>[r]^{\ \  \quad \quad \quad \ \ \ \ \ \ h}\ar@<-2pt>[r]_{ \  \quad \quad \quad \ \ \ \ \ \  g} &   \Tot( \EE nd (\FF \oplus \GG)(\VV))
      \ar@<4pt>[r] \ar[r] \ar@<-4pt>[r] & 0}
\end{equation}
and then apply Thom-Withney homotopy fibre product. This dgLa is quasi isomorphic to $\Tot(H(\VV)^\Delta)$.
\end{remark}

Since the cohomology of the total complex of the semicosimplicial dgLa $H(\VV)^\Delta$ is the same as the one of the Thom-Whitney homotopy fibre, according to \cite[Section 4]{Dona.def maps}, we get the following exact sequence: 
\[  \ldots \to H^i(\Tot(H(\VV)^\Delta) \to H^i( X, \EE nd (\FF) \oplus \EE nd (\GG) \oplus \LL) \to \]
\[\to H^i(X, \EE nd (\FF \oplus \GG)) \to H^{i+1}(\Tot(H(\VV)^\Delta) \to  \ldots  \]
Moreover,
since the map  $h:    \LL   \to    \EE nd (\FF \oplus \GG) $ is injective,
thanks to \cite[Lemma 3.1]{Dona.def maps},  the following sequence is also exact:
\begin{equation} \label{succ ex lunga generale}
 \cdots H^i(\Tot(H(\VV)^\Delta) \to H^i( X, \EE nd (\FF) \oplus \EE nd (\GG)) \to H^i(X,\coker h)    \to H^{i+1}(X,\Tot(H(\VV)^\Delta)  \cdots .
\end{equation}

Note that, since $\Tot(H(\VV)^\Delta)$ controls   the infinitesimal deformations of $\alpha:\FF \to \GG$, its first  cohomology space is isomorphic to the tangent space of the functor $\Def_{(\FF, \alpha, \GG)}$, and its second cohomology space is an obstruction space.

\begin{remark} \label{rmk.sui numeri complessi}
If we work over the field of complex numbers, we can consider the  Dolbeault resolutions of the endomorphisms of sheaves, and avoid the \v{C}ech semicosimplicial dgLas. More precisely, we can consider the two morphisms of dgLas
  \[  
\xymatrix{ { A_X^{0,*}(\LL)\oplus  A_X^{0,*}( \EE nd(\FF)) \oplus A_X^{0,*}(\EE nd (\GG)) }
\ar@<2pt>[r]^{\ \  \quad \quad \quad   h}\ar@<-2pt>[r]_{ \  \quad \quad \quad  g} &  A_X^{0,*}(\EE nd (\FF \oplus \GG)),}
\]  analogous to the previous $h$ and $g$.
Then, the associated Thom-Whitney homotopy fibre product of Definition  \ref{defi thom with coppia morfismi} is a dgLa that controls infinitesimal deformations of $\alpha: \FF \to \GG$. Note that this construction is the analogous of the one in Equation \eqref{successione nella remark} of   Remark \ref{remark thom whit non dipende da verso}.

\end{remark}

\section{Deformations of a locally free sheaf and a subspace of sections}\label{sezione deformazioni (E,U)}
In this section, we would like to apply the previous results to    the infinitesimal deformations of a locally free sheaf with a subspace of its sections, finding a dgLa that controls them.

\subsection{Geometric deformations.}

Let $X$ be a smooth projective variety over $\mathbb{K}$, $\EE$ be  a locally free sheaf of $\OO_X$-modules on $X$ and $U\subseteq H^0(X,\EE)$ be  a subspace of global sections of $\EE$.
We study the infinitesimal deformations of $\EE$ which preserve the subspace $U$.
We start with some definitions. 

\begin{definition} \label{def.(E,U)}
Let $A\in \Art_\kk$. An \emph{infinitesimal deformation} of the pair $(\EE, U)$  over $A$  is the data $(\EE_A,\pi_A, i_A)$ of:
\begin{itemize}
\item a deformation $(\EE_A, \pi_A)$ of $\EE$ over $A$,
\item a morphism $i_A:U\otimes A \rightarrow H^0(X\times \Spec A,  \EE_A)$,
\end{itemize}
such that the following diagram commutes
\begin{equation}  \label{def(E,U)}
\xymatrix{ U\otimes A \ar[d]^{\pi} \ar[r]^-{i_A} & H^0(X\times \Spec A,\EE_A) \ar[d]^{\pi_A}\\
U \ar@{^{(}->}[r]^-{i} & H^0(X, \EE).     }
\end{equation}

Two deformations  $(\EE_A,\pi_A,i_A)$, $(\EE'_A,\pi'_A,i'_A)$ are \emph{isomorphic} if there exist an isomorphism $\phi:\EE_A \rightarrow \EE'_{A}$ of sheaves of $\OO_X \otimes A$-modules,  such that
 $\pi'_A \circ \phi=\pi_A$ 
and an isomorphism $\psi:U\otimes A \to U\otimes A$,  that makes the diagram commutative:
\[
\xymatrix{ U\otimes A \ar[d]^{\psi}   \ar[r]^-{i_A}& H^0(X\times \Spec A,\EE_A) \ar[d]^{\phi} \\
             U\otimes A \ar[r]^-{i'_A}& H^0(X\times \Spec A,\EE'_A) . }
\]

Note that, this implies that $\phi$ induces an isomorphism  $\phi: i_A(U\otimes A) \rightarrow i'_A(U\otimes A)$. In the following, we will often shorten the notation of such deformations with $(\EE_A, i_A)$. 
\end{definition}

\begin{definition} 
The \emph{functor of infinitesimal deformations of $(\EE,U)$} is
\[ \Def_{(\EE,U)}: \Art_\kk \to \Grpds,\]
that associates to any $A \in \Art_\kk$ the groupoid $\Def_{(\EE,U)}(A)$, whose objects are the infinitesimal deformations of the pair $(\EE, U)$ over  $A$ and whose morphisms are the isomorphisms among them.
 Note that  
\[ \pi_0(\Def_{(\EE, U))}(A)) = \{\mbox{isomorphism classes of deformations of the pair $(\EE, U)$ over $A$}\} \]
is the classical functor of deformations in $\Set$.
\end{definition}

\begin{remark} \label{rmk.equivalenza def coppia e morfismo}
Let $X$ be a smooth projective variety over $\mathbb{K}$ with $H^1(X, \OO_X)=0$. Let $\EE$ be  a locally free sheaf of $\OO_X$-modules on $X$ and $U\subseteq H^0(X,\EE)$   a subspace of global sections of it. Let $s: U \otimes \OO_X \to \EE$ be the morphism of sheaves of $\OO_X$-modules associated to $i: U \hookrightarrow H^0(\EE)$ via the obvious correspondence between global sections and morphisms from $\OO_X$ to $\EE$.

According to Definition \ref{def.funtore def in gruppoidi di alpha}, we denote by
 \[ \Def_{(U\otimes \OO_X,s, \EE)} : \Art_\kk \to \Grpds,\]
the functor of infinitesimal deformations of the morphism $s$, 
that associates, to any $A \in \Art_\kk$, the groupoid $\Def_{(U\otimes \OO_X,s, \EE)}(A)$, whose objects are the infinitesimal deformations of the morphism $s$  over
$A$ and whose morphisms are the isomorphisms  between them.

We claim that, for any $A\in \Art_\kk$ there is a 1-1 correspondence between the objects  of the groupoids 
\[ \Def_{(\EE, U)} (A)\to \Def_{(U\otimes \OO_X,s, \EE)}(A).\]
Indeed, 
given a deformation $(\EE_A, i_A)$ of $(\EE, U)$ over $A$ as in Definition \ref{def.(E,U)}, there is just one morphism of sheaves of $\OO_X\otimes A$-modules $s_A:U\otimes \OO_X\otimes A \to \EE_A$ associated to  $i_A$, it makes the following diagram commutative
\begin{equation} \label{def s}  
\xymatrix{ U\otimes \OO_X \otimes A \ar[d]^{\pi} \ar[r]^-{s_A} &\EE_A \ar[d]^{\pi_A}\\
U \otimes \OO_X \ar[r]^-{s} &\EE,    }
\end{equation}
and it defines an object of $\Def_{(U\otimes \OO_X, s, \EE)}(A)$.
On the other hand, since $H^1(X, \OO_X)=0$, every infinitesimal deformation of the domain $U\otimes \OO_X$  is trivial. A deformation of $s$ over $A$ is just the data of a deformation $\EE_A$ of $\EE$ and of a map   $s_A:U\otimes \OO_X\otimes A \to \EE_A$, such that the diagram (\ref{def s}) is commutative.
Taking global sections, one gets the required deformed map $i_A: U\otimes A \to H^0(X \times \Spec A, \EE_A)$ and the commutative diagram (\ref{def(E,U)}).

Analogously, we can prove that isomorphisms of deformations correspond each others. This assures that, under the hypothesis $H^1(X, \OO_X)=0$, the functors $\Def_{(\EE, U)}$ and $\Def_{(U\otimes \OO_X,s, \EE)}$ are  are equivalent.
\end{remark}

In \cite{elenatesi, ElenaFormal}, the second author introduced the problem  of deformations of the pair $(\EE,U)$ and found a dgLa that controls it on a variety $X$ over the field of complex numbers.
In \cite{DE Brill-Noether}, we studied these deformations  in more details and  we generalised to vector bundles over a smooth complex projective variety   some classical results known  for line bundles over curves, concerning the description of the tangent space, the smoothness of the functor $\Def_{(\EE,U)}$ and  the liftings of a section.  Here, we would like to study  the infinitesimal deformations of the pair $(\EE,U)$ over  an algebraically closed field  $\mathbb{K}$ of characteristic zero using the analysis of the previous sections. 
 
\subsection{The dgLa that controls deformations of $(\EE,U)$}
Let $X$ be a smooth projective variety over $\mathbb{K}$, we assume here $H^1(X, \OO_X)=0$. Let $\EE$ be  a locally free sheaf of $\OO_X$-modules on $X$ and $U\subseteq H^0(X,\EE)$ be  a subspace of global sections of it.

 Let $s: U \otimes \OO_X \to \EE$ be the morphism of sheaves of $\OO_X$-modules associated to $i: U \hookrightarrow H^0(\EE)$.
As observed in Remark \ref{rmk.equivalenza def coppia e morfismo}, the infinitesimal deformations of the pair $(\EE, U)$ are equivalent to the infinitesimal deformations of the morphism $s$. We indicate with $\gamma$ the graph of the morphism $s$, that is a subsheaf of $(U\otimes \OO_X) \oplus \EE$.

Let $\VV=\{  V_i \}_i$ be an affine open cover of $X$, we denote by $ \EE nd (\EE), \EE nd(U\otimes \OO_X)$ and $\EE nd \left((U\otimes \OO_X) \oplus \EE\right)$ the sheaves of $O_X $-modules endomorphisms of $\EE, U\otimes \OO_X$ and $(U\otimes \OO_X) \oplus \EE$, respectively. 
Let $H(\VV)^\Delta$ be  the Cech semicosimplicial dgLa associated to
the Thom-Whitney homotopy fibres of the diagram of sheaves of Lie algebras
\[ \xymatrix{  & \LL \ar[d]^h \\ \EE nd(U\otimes \OO_X) \oplus \EE nd (\EE) \ar[r]^-g
 &   \EE nd \left((U\otimes \OO_X) \oplus \EE\right),} \]
where  $h: \LL = \{ \varphi \in \EE nd \left((U\otimes \OO_X) \oplus \EE\right) \mid \varphi(\gamma) \subseteq \gamma \} \to   \EE nd \left((U\otimes \OO_X) \oplus \EE\right)$  is the inclusion and $g$ is defined as the map 
$(\varphi, \psi) \mapsto\left( \begin{array}{cc} \varphi & 0 \\ 0 & \psi\end{array}\right).
$
Applying the results of the previous section, we get the following result

\begin{theorem} \label{thm. def (E,U) usando tot} 
   In the above notation and under the hypothesis $H^1(X,\OO_X)=0$, the functor $\Def_{(\EE, U)}$ of deformations of $(\EE,U)$ is equivalent to the Deligne functor associated to the Thom-Whitney dgLa of the semicosimplicial dgLa $H(\VV)^\Delta$.
Moreover, under the additional hypothesis that $H^2(X,\OO_X)=0$, the cohomology of the dgLa $\Tot(H(\VV)^\Delta)$ fits into the exact sequence 
\begin{equation} \label{eq.theorem}
\begin{split} 
0 \to H^0(\Tot(H(\VV)^\Delta)) \to H^0(X, \EE nd (\EE)) 
 \to \Hom(U, H^0(\EE)/U) \to \\ \to H^1(\Tot(H(\VV)^\Delta)) \to H^1(X,\EE nd (\EE))   \stackrel{\alpha}{\to}  
 \Hom (U, H^{1}(X,\EE))  \to \\ \to  H^2(\Tot(H(\VV)^\Delta))  	\to  H^2(X,\EE nd (\EE)) \to\Hom(U, H^2(X,\EE)).
 \end{split}
\end{equation}
\end{theorem}
 
 \begin{proof} 
The equivalence
\[ \Def_{(\EE,U)} \cong \Del_{\Tot(H(\VV)^\Delta)} \]
is a direct consequence of Theorem \ref{thm.Del eq Def}  and Remark \ref{rmk.equivalenza def coppia e morfismo}. 
According to the sequence  \eqref{succ ex lunga generale}, the cohomology of the dgLa $\Tot(H(\VV)^\Delta)$ fits in the exact sequence 
\begin{equation}\label{succ ex 1}
 \cdots H^i(\Tot(H(\VV)^\Delta) \to H^i( X, \EE nd (U\otimes \OO_X) \oplus \EE nd (\EE)) \to H^i(X, \coker h)    \to H^{i+1}(\Tot(H(\VV)^\Delta)  \cdots.
\end{equation}
Observe that, since $X$ is projective, 
\[ H^0( X, \EE nd(U\otimes \OO_X) \oplus \EE nd (\EE))= \Hom(U,U) \oplus H^0(X, \EE nd (\EE)); \] 
while for $i=1,2$, under the hypothesis $H^i(X, \OO_X)=0$, we get
\[   H^i( X, \EE nd (U\otimes \OO_X) \oplus \EE nd (\EE))=\left(H^i(X, \OO_X)\otimes U^\vee \otimes U\right) \oplus  H^i(X, \EE nd (\EE) =H^i(X, \EE nd (\EE)).\]
Moreover, we can describe explicitly the $\coker h$ via the following isomorphisms
\[  \coker h \cong \HH om\left(\gamma, \frac{(U\otimes \OO_X) \oplus \EE}{\gamma}\right) \cong \HH om(U\otimes \OO_X, \EE). \]
The first one is  given by restricting an endomorphism of $(U\otimes \OO_X) \oplus \EE $ to $\gamma$ and projecting it to the quotient. The second one is obtained just observing that $U\otimes \OO_X \cong \gamma$ and $\displaystyle\frac{(U\otimes \OO_X) \oplus \EE}{\gamma} \cong \EE$. 
Thus, for $i=0$, we have
\[ H^0(\coker h)= H^0(  \HH om(U\otimes \OO_X, \EE)) =\Hom(U, H^0(X, \EE))= \Hom(U, U) \oplus \Hom\left(U, \frac{H^0(X,\EE)}{U}\right), \]
where the last equality follows applying $\Hom(U, -)$ to the exact sequence of vector spaces
\[ 0\to U \to H^0(X, \EE) \to H^0(X,\EE)/U \to 0. \]
While, for $i=1,2$, we have
\[H^i(\coker h)= H^i(\HH om(U\otimes \OO_X, \EE)) = \Hom(U, H^i(X, \EE)).\]
Thus, the first terms of the exact sequence (\ref{succ ex 1}) become
\begin{equation}  \label{succ ex 2}
0 \to H^0(\Tot(H(\VV)^\Delta)) \to \Hom(U,U)  \oplus H^0(X, \EE nd (\EE)) \stackrel{f}{\to}  
  \Hom(U, U) \oplus \Hom\left(U, \frac{H^0(X,\EE)}{U}\right) \to \
  \end{equation}
\[
\to H^1(\Tot(H(\VV)^\Delta)) \to H^1(X,\EE nd (\EE))  \to 
 \Hom (U, H^{1}(X,\EE))  
  \to  H^2(\Tot(H(\VV)^\Delta))  	\to  \]
  \[\to H^2(X,\EE nd (\EE)) \to\Hom(U, H^2(X,\EE)).\]
The last step is to make explicit the morphism $f$ of (\ref{succ ex 2}). It is immediate to see that $f$ is the morphism
\[ \Hom(U,U) \oplus H^0(X, \EE nd (\EE)) \stackrel{f}{\to}  
  \Hom(U, U) \oplus \Hom\left(U, \frac{H^0(X,\EE)}{U}\right)\!,  \]
whose first component is the identity on $ \Hom(U, U)$. Thus the exact sequence (\ref{succ ex 2}) induces the sequence \eqref{eq.theorem}. 
\end{proof}

 The exact sequence \eqref{eq.theorem}  of Theorem \ref{thm. def (E,U) usando tot}  is a generalisation over any algebraically closed field of characteristic zero of the exact sequence (7) in \cite{DE Brill-Noether},  from which we get most of our results, that now can be immediately generalised to    any  algebraically closed field $\kk$ of characteristic zero.  

Let us denote by $r_U: \Def_{(\EE,U)} \to \Def_\EE$, the forgetful functor, i.e., the functor that ignores all information about $U$. Then, under the hypothesis $H^i(X,\OO_X)=0$  for $i=1,2$, we have the following results, that holds for over  any  algebraically closed field $\kk$ of characteristic zero.

\begin{corollary} \label{corollario alpha su implica r liscio} \cite[Corollary 3.11]{DE Brill-Noether}
If the map $\alpha: H^1(X,\EE nd (\EE)) 
\to \Hom (U, H^{1}(X,\EE))$ that appears in the exact sequence \eqref{eq.theorem} is surjective or, that is equivalent, if $\Hom(U, H^1(X,\EE))=0$, then  the forgetful morphism $r_U: \Def_{(\EE,U)} \to \Def_\EE$ is smooth. \end{corollary}

\begin{corollary}\label{corcor}  \cite[Corollary 3.13]{DE Brill-Noether} In the notation above, we have 
 \[ \dim t_{\Def_{(\EE,U)}} \geq \dim t_{\Def_{\EE}}
 - m \cdot \dim H^1(X,\EE),\]
 where $m$ is the dimension of $U \subseteq H^0(X,\EE)$.  
\end{corollary}

\begin{corollary} \label{prop.sezione si estende se cup product va a zero} \cite[Corollary 3.11]{DE Brill-Noether}  A section $s  \in H^0(X,\EE)$ can be extended to a section of a first order deformation of $\EE$ associated to an element  $a \in H^1(X, \EE nd (\EE))$ if and only if $a \cup s =0 \in H^1(X,\EE)$.
\end{corollary}

\begin{corollary} \label{Cor. TDef(E,H0(E))} \cite[Corollary 3.15]{DE Brill-Noether}
The tangent space to the deformations of the pair $(\EE, H^0(\EE))$ can be identified with
\[ t_{\Def_{(\EE,H^0(\EE))}} = \{ a \in H^1(X,\EE nd (\EE)) \ \mid \ a \cup s =0, \ \forall  \ s \in H^0(X,\EE) \}.  \]
\end{corollary}

\begin{remark}
We  underline that the results of \cite{DE Brill-Noether} hold over the field of complex numbers $\CC$
without the additional hypothesis that $H^i(X,\OO_X)=0$  for $i=1,2$.
\end{remark}

\subsection{ An explicit model for the semicosimplicial dgla that controls deformations of $(\EE, U)$.}\label{sezione dgla esplicit model}   In this section, we describe another semicosimplicial dgLa whose total complex controls deformations of the pairs $(\EE,U)$ over a field $\kk$.   
The construction is more explicit and it avoids the Theorem on descent of Deligne groupoids (Theorem \ref{hinich theorem}).
 
In the previous notation, given the morphism of sheaves $s: U \otimes \OO_X \to \EE$, we consider the complex of sheaves
\[{\mathcal Q}: U\otimes \OO_X \stackrel{s}{\to} \EE, \]
 where the sheaves $U\otimes \OO_X$ and $\EE$ are settled in degree zero and one, respectively. As above, $s$ is the morphism associated to $i: U \hookrightarrow H^0(\EE)$ and it is a differential, $s^2=0$, since there is no degree 2. 
A similar construction is used in \cite[pag. 234-235]{Manetti.LMDT} to analyse the embedded deformations. Consider the following  complex of sheaves concentrated in degree 0 and 1: 
\[ \HH om^{\geq 0}({\mathcal Q},{\mathcal Q}) = \left\{ \begin{array}{cll} 
\HH om^0({\mathcal Q},{\mathcal Q}) &=&  \EE nd (U\otimes \OO_X) \oplus  \EE nd(\EE) \\
\HH om^1({\mathcal Q},{\mathcal Q}) &= & \HH om(U\otimes \OO_X, \EE) \\
\HH om^k({\mathcal Q},{\mathcal Q}) &= & 0 \quad \forall \ k \neq 0,1,
\end{array}     \right. \]
This is a sheaf of dgLas. The differential is given by $s$, the bracket is the usual one in $\EE nd (U\otimes \OO_X)$ and $  \EE nd(\EE)$, it is the trivial bracket between an element in  $\EE nd (U\otimes \OO_X)$  and an element in $  \EE nd(\EE)$ and it is given by the composition between elements in degree 0 and 1. 

Next, fix an open affine cover $\VV=\{  V_i \}_i$ of $X$ trivialising $\EE$ and let $\g^\Delta$ be the \v{C}ech semicosimplicial dgLa  $\g^\Delta$ associated to the sheaf $ \HH om^{\geq 0}({\mathcal Q},{\mathcal Q})$ and to the open cover $\VV$. 
The aim of this section is to prove the following result.

\begin{theorem} \label{teorema H1 e Def(E,U)}
The \v{C}ech semicosimplicial dgLa  $\g^\Delta$ associated to the sheaf $ \HH om^{\geq 0}({\mathcal Q},{\mathcal Q})$ controls the infinitesimal deformations of the morphism $s: U\otimes \OO_X \to \EE$, i.e., 
the functor $\widetilde{H}^1_{sc}(\exp \g^{\Delta})$ and the functor $\Def_{(U\otimes \OO_X, s, \EE)}$ are isomorphic. 
\end{theorem}

\begin{proof}
For any $A \in \Art_\kk$, we define a 1-1 correspondence between the objects of the groupoids
\[ \widetilde{H}^1_{sc}(\exp \g^{\Delta})(A)\to \Def_{(U\otimes \OO_X, \EE, s)}(A). \]  

 An element in $Z^1_{sc}(\exp \g^{\Delta})(A)$ is a pair $(l,m) \in\left(\g_0^1 \oplus \g_1^0\right)\otimes \m_A$, where  the  element $l=\{l_i\}_i \in \prod _i \HH om(U\otimes \OO_X, \EE)(V_i) \otimes \m_A$ and the element 
$m=\{\nu_{ij}, \mu_{ij}\}_{ij} \in \prod _{i<j} \left(\EE nd(U\otimes \OO_X)(V_{ij})\oplus \EE nd(\EE)(V_{ij})\right) \otimes \m_A$, and it has to satisfy the equations in (\ref{def Z1 caso gradi non neg}).

Using these data, we can define a  deformation of the morphism  $s: U\otimes \OO_X  \to \EE$ over $A$ as follows.
\begin{itemize}
\item The deformation $\EE_A$ of the locally free sheaf $\EE$ over $A$ is obtained by gluing the local trivial deformations $\{\EE|_{V_i}\otimes A\}_i$  by the isomorphisms  $\{e^{ \mu_{ij}}: \EE_{|_{V_{ij}}}\otimes A \to \EE_{|_{V_{ij}}}\otimes A\}_{ij}$. 
Note that the third equation of (\ref{def Z1 caso gradi non neg}) gives rise in our case to the equation ${\partial_{0,2} \mu} \bullet {-\partial_{1,2} \mu} \bullet
{\partial_{2,2} \mu} = 0$, that assures that the given gluing  functions are compatible on the triple intersections. 
\item Similarly, the data $\{e^{\nu_{ij}}: U\otimes \OO_X|_{V_{ij}} \otimes A \to U\otimes \OO_X|_{V_{ij}} \otimes A\}_{ij}$ are gluing isomorphisms of the local trivial deformations of the sheaf $U\otimes \OO_X$ over $A$ in the open sets ${V_{ij}}$. The third equation in  (\ref{def Z1 caso gradi non neg})  gives rise to the equation  ${\partial_{0,2}\nu} \bullet {-\partial_{1,2}\nu} \bullet
{\partial_{2,2}\nu} = 0$ and it assures that  the gluing isomorphisms are compatible on the triple intersections. 
\item  The morphism $s_A: U\otimes \OO_X \otimes A  \to \EE_A$ is locally defined as $\{s|_{V_i} + l_i: U\otimes \OO_X|_{V_{i}}\otimes A \to \EE|_{V_i} \otimes A\}_i$. Note that the second equation  of  (\ref{def Z1 caso gradi non neg})  is $\delta_{11} l = e^m * \delta_{01} l$  and it can be rewritten as $s + \delta_{11} l = e^\mu (s + \delta_{01} l ) e^{-\nu}$ (see for example \cite[Formula 8]{DE Brill-Noether}).  It says that the local maps $\{s|_{V_i} + l_i\}_i$ glue to a global one.
\end{itemize}
Note that the first equation $dl + \frac{1}{2}[l,l]=0$ is trivial in our case, because our dgLas are zero in degree $2$. 

Thus every element in  $Z^1_{sc}(\exp \g^{\Delta})(A)$ exactly defines an infinitesimal deformation of the morphism $s:U\otimes \OO_X \to \EE$ on $A$. Let now $(l_0,m_0), (l_1,m_1) \in  Z^1_{sc}(\exp \g^{\Delta})(A)$ be isomorphic elements. We prove that  they define two isomorphic deformations in $\Def_{(U\otimes \OO_X, s,\EE)}(A)$.
\begin{itemize}
\item By the first equation in  (\ref{def rel equiv caso gradi non neg}), there exists $a \in \g_0^0 \otimes \m_A$ such that $e^a *l_0=l_1$. Note that  
$a =(\beta,\alpha )= \{ \beta_i, \alpha_i \}_i   \in \prod_i \left( \EE nd(U\otimes \OO_X)(V_i) \oplus \EE nd(\EE)(V_i) \right) \otimes \m_A$ 
and as usual, the equation can be rewritten as $s +l_1 =e^\alpha (s +l_0) e^{-\beta} $. Then,  for every $i$,   $e^{\alpha_i}: \EE_{|_{V_i}} \otimes A\to \EE_{|_{V_i}} \otimes A $ and $e^{\beta_i}: U\otimes {\OO_X}|_{V_i} \otimes A \to U\otimes {\OO_X}|_{V_i} \otimes A$ define local isomorphisms that make the following diagram commutative: 
\[
\xymatrix{ U\otimes \OO_X|_{V_i} \otimes A \ar[d]^{e^{\beta_i}}   \ar[r]^-{s+ (l_0)_i}& \EE|_{V_i}\otimes A \ar[d]^{e^{\alpha_i}} \\
             U\otimes \OO_X|_{V_i} \otimes A \ar[r]^-{s + (l_1)_i}&  \EE|_{V_i}\otimes A . }
\]

\item The second equation  in (\ref{def rel equiv caso gradi non neg}) gives rise to the following two equations $ - \mu_0\bullet -\partial_{1,1}\alpha \bullet \mu_1
\bullet \partial_{0,1}\alpha=0$ and $- \nu_0\bullet -\partial_{1,1}\beta \bullet \nu_1
\bullet \partial_{0,1}\beta=0$ that assure that the isomorphisms  $e^{\alpha_i}$ and  $e^{\beta_i}$ glue to a global one. 
\end{itemize}
Thus, as wanted, the two deformations defined by $(l_0,m_0), (l_1,m_1)$ are isomorphic. 

In this way, we obtain every infinitesimal  deformation of the morphism $s:U\otimes \OO_X \to \EE$. Indeed, on a fix open affine set $V_i$,  the infinitesimal deformations of $\EE$ and $U\otimes \OO_X$ are trivial, while a deformation of $s:U \otimes \OO_X \to \EE$ over $V_i$, is given by $s |_{V_i} + l_i: U \otimes \OO_X|_{V_{i}}\otimes A \to \EE|_{V_i}\otimes A $, for $l_i \in  \HH om(U\otimes \OO_X, \EE)(V_i) \otimes \m_A$.
Then, we have to glue these data along double intersections to give a deformation of the morphism $s$ on $X$.  Therefore, we need isomorphisms $ \tilde{\mu}_{ij}: \EE_{|_{V_{ij}}}\otimes A \to \EE_{|_{V_{ij}}}\otimes A $ for all $i, j$, that satisfy the cocycle condition on triple intersections: $ \tilde{\mu}_{jk}\circ  {\tilde{\mu}_{ik}}^{-1}\circ  \tilde{\mu}_{ij}=\Id$, to glue the trivial deformations of $\EE$ to a global one.  Moreover, we need  isomorphisms $ \tilde{\nu}_{ij}: U\otimes \OO_X|_{V_{ij}}\otimes A \to U\otimes \OO_X|_{V_{ij}}\otimes A $ such that the following diagram commutes
\[
\xymatrix{ U\otimes \OO_X|_{V_i} \otimes A \ar[d]_{\tilde{\nu}_{ij}}   \ar[rr]^-{{( s |_{V_i} + l_i)}_{|_{V_{ij}}}}& & \EE|_{V_i}\otimes A \ar[d]^{ \tilde{\mu}_{ij}} \\
           U\otimes \OO_X|_{V_j} \otimes A   \ar[rr]_ -{{( s |_{V_j} + l_j)}_{|_{V_{ij}}}}& & \EE|_{V_j}\otimes A   , }
\]
i.e., $  {( s |_{V_j} + l_j)}_{|_{V_{ij}}} \circ \tilde{\nu}_{ij}  =  \tilde{\mu}_{ij}\circ {( s |_{V_i} + l_i)}_{|_{V_{ij}}}.$  Then, the local trivial deformations will glue to a global non trivial deformation of $s$.  
 
Since we are in characteristic zero, we can take the logarithm to obtain $(\tilde{\mu}_{ij}, \tilde{\nu}_{ij} )= (e^{ \mu_{ij}}, e^{ \nu_{ij}})$ with 
 $\{(\nu_{ij}, \mu_{ij})\}_{ij} \in\prod _{i<j}  \left(  \EE nd( U\otimes \OO_X)(V_{ij})\oplus \EE nd( \EE)(V_{ij})\right) \otimes \m_A$.

Therefore, any deformation of the morphism $s$ over $A$ is given by an element $(l,m) \in\left(\g_0^1 \oplus \g_1^0\right)\otimes \m_A$, where $l=\{l_i\}_i \in  \prod _{i}\HH om(U\otimes \OO_X, \EE)(V_i) \otimes \m_A$ and the element 
$m=\{\nu_{ij}, \mu_{ij}\}_{ij} \in \prod _{i<j}  \left(   \EE nd(U\otimes \OO_X) (V_{ij})\oplus \EE nd( \EE)(V_{ij})\right) \otimes \m_A$,
that has to satisfy the conditions in (\ref{def Z1 caso gradi non neg}).

Analogously, we can prove that isomorphisms of deformations correspond to the existence of an element $a \in \g^0_0 \otimes \m_A $ that satisfy the condition in Equation  \eqref{def rel equiv caso gradi non neg}. 
\end{proof}

As a direct consequence of Remark \ref{rmk.equivalenza def coppia e morfismo}, Theorems \ref{teorema H1 e Def(E,U)} and  \ref{teo.H1sc def functors}, we get the following result.

\begin{corollary}
Under the hypothesis $H^1(X, \OO_X)=0$,  there are  natural isomorphisms of functors
\[
 \Def_{(\EE,U)}   \cong  \Def_{(U\otimes \OO_X, s, \EE)} \cong H^1_{sc}(\exp \g^\Delta) \cong \Def_{\Tot(\g^{\Delta})}.
\]
\end{corollary}
This corollary provides another explicit description of a dgLa controlling the infinitesimal deformations of the pair $(\EE,U)$.

\bigskip

Fix an open affine  cover $\{ V_i \}_{i}$ of $X$. As above, let $\g^\Delta$ and $\h^\Delta$ be the \v{C}ech semicosimplicial  dgLas associated to the sheaf   $\HH om^{\geq 0}({\mathcal Q}, {\mathcal Q})$ and to the sheaf $\EE nd(\EE)$, respectively. There is an obvious surjective morphism of semicosimplicial dgLas
$ \g^\Delta \to \h^\Delta$ and, denoting with $\m^\Delta$ the kernel of it, we get the following exact sequence of semicosimplicial dgLas:
\begin{equation} \label{ex.seq.} 0 \to \m^\Delta \to \g^\Delta \to \h^\Delta \to 0. \end{equation}

\begin{lemma} \label{prop.cohomology of mdelta} 
Under the hypothesis that  $H^i(X,\OO_X)=0$ for $i=1,2$, 
the cohomology of the total complex of the semicosimplicial dgLa $\m^\Delta$ is given by: 
\[ H^0\left(\Tot(\m^\Delta)\right)=0,  H^1\left(\Tot(\m^\Delta)\right) = \Hom\left(U, H^0(\EE)/U\right) \ \mbox{and}\ H^2\left(\Tot(\m^\Delta)\right) = \Hom(U, H^1(\EE)). \]
\end{lemma}

\begin{proof}
The double complex associated to the semicosimplicial dgLa $\m^\Delta$ is
\begin{equation}  \label{doppio complesso}
\xymatrix@C=0.8em{
 {\prod_i\HH om(U\! \otimes \! \OO_X \!, \EE) (V_i) } \ar[r]^-{{\check{\delta}}} & {\prod_{i<j} \HH om(U\! \otimes \! \OO_X\!, \EE) (V_{ij}) } \ar[r]^-{\check{\delta}} & {\prod_{i<j<k} \HH om(U\! \otimes \! \OO_X\!, \EE) (V_{ijk}) }\ar[r]^-{\check{\delta}}  & \ldots \\
{\prod_i  \EE nd(U\! \otimes \! \OO_X)  (V_i) } \ar[r]^-{\check{\delta}}\ar[u]^{d}  & {\prod_{i<j}   \EE nd(U\! \otimes \! \OO_X) (V_{ij}) }\ar[r]^-{\check{\delta}} \ar[u]^{d} & {\prod_{i<j<k}  \EE nd(U\! \otimes \! \OO_X)  (V_{ijk}) }\ar[r]^-{\check{\delta}} \ar[u]^{d} & \ldots   
}   \end{equation}
where the horizontal differential $\check{\delta}$ is the \v{Cech} one, while the vertical $d$ is the differential of the dgLa involved. 

Let $\{E_k^{p,q}\}$ be the spectral sequence associated to the double complex (\ref{doppio complesso}). Calculating the \v{C}ech cohomology of it, we get the first page:
\begin{equation}  \label{prima pagina} 
\xymatrix@C=0.3em{ 
{\Hom(U, H^0(X,\EE)) } & { \Hom (U, \check{H}^1(X,\EE)) }  &  { \Hom (U, \check{H}^2(X,\EE)) }& \ldots &  { \Hom (U, \check{H}^k(X,\EE))}  \\
{ \Hom(U, U) }\ar[u]^{d}  & {0 } \ar[u]^{d} & {0}\ar[u]^{d}  &\ldots& H^k(  X , \OO_X)\otimes \Hom(U,U) \ar[u]^{d}   \\  }   \end{equation}
where we use that  $\check{H}^0( X,  \EE nd(U\otimes \OO_X))= \Hom (U,U)$ and  $ \check{H}^k( X,  \EE nd(U\otimes \OO_X))= 0$ for $k=1,2$ and $\check{H}^k(X,\HH om(U\otimes \OO_X, \EE))= \Hom (U, \check{H}^k(X,\EE))$ for all $k\geq 0$. 
Since the differential from the second page is always zero, the spectral sequence abuts to $E_2^{p,q}=E_{\infty}^{p,q}$.
This gives the expected cohomology of the total complex $\Tot(\m^\Delta)$. 
\end{proof}

  From the exact sequence \eqref{ex.seq.} and thanks to Lemma \ref{prop.cohomology of mdelta}, we recover  the  exact sequence \eqref{eq.theorem}:
\begin{equation} 
\begin{split} \label{ex seq cohomology}
0 \to H^0(\Tot(\g^\Delta)) \to H^0( X,\EE nd (\EE))  
 \to \Hom (U, H^0(X,\EE)/U  )\\  \to H^1(\Tot(\g^\Delta)) \to H^1(X,\EE nd (\EE))   \stackrel{\alpha}{\to}  
   \Hom (U, H^{1}(X,\EE))  \stackrel{\beta}{\to}  \\  
   \to H^2(\Tot(\g^\Delta))  \stackrel{\gamma}{\to}  H^2(X,\EE nd (\EE)) \to \Hom (U, H^{2}(X,\EE)). \\
 \end{split}
\end{equation}
 Thus, the two dgLas $\Tot(H({\mathcal V})^\Delta)$ and $\Tot(\g^\Delta)$, constructed to control the infinitesimal deformations of the pair $(\EE, U)$, have  actually   the same tangent and obstructions space.

\begin{remark}
Note that, if we strengthen the hypothesis of Theorem \ref{thm. def (E,U) usando tot}  and of Lemma \ref{prop.cohomology of mdelta}, assumining that $H^i(X, \OO_X)=0$ for all $i>0$, both the exact sequences \eqref{eq.theorem} and \eqref{ex seq cohomology} continue in higher degrees giving rise to the exact sequences: 
\[
\ldots \to H^i(\Tot(\g^\Delta)) \to H^i( X,\EE nd (\EE))  
 \to \Hom (U, H^i(X,\EE)  ) \to H^{i+1}(\Tot(\g^\Delta)) \to \ldots 
\]
and \[
\ldots \to H^i(\Tot(H({\mathcal V})^\Delta) \to H^i( X,\EE nd (\EE))  
 \to \Hom (U, H^i(X,\EE)  ) \to H^{i+1}(\Tot(H({\mathcal V})^\Delta)) \to \ldots 
\] 
respectively. Thus, in this case the two dgLas $\Tot(H({\mathcal V})^\Delta)$ and $\Tot(\g^\Delta)$ are quasi-isomorphic.
\end{remark}

\section{Deformations of a locally free sheaf preserving some sections}\label{sezione finale defo con sezioni}

In this section, we consider a locally free sheaf $\EE$ of $\OO_X$-modules on a smooth projective variety $X$ over $\kk$, such that $\dim H^0(X,\EE) \geq k$. We analyse the infinitesimal deformations of $\EE$ such that at least $k$ independent sections of $\EE$ lift to the deformed locally free sheaf.  Our approach is the same as in \cite{DE Brill-Noether} improved with the techniques developed   in the previous sections that allow to generalise the results to   any algebraically closed field $\kk$ of characteristic zero.
 
We start giving the following definition.
\begin{definition}
Let $\EE$ be a locally free sheaf of $\OO_X$-modules on a smooth projective variety $X$ over $\kk$, such that $h^0(X,\EE) \geq  k$.
Let $Gr(k,H^0(X, \EE))$ be the Grassmannian of all subspaces of $H^0(X, \EE)$ of dimension $k$. We define the functor $\Def^{k}_\EE:\Art_\kk \to \Set$, that associates to every $A \in \Art_\kk$  the set
\[ \Def^{k}_\EE (A)= \bigcup_{U \in Gr(k,H^0(X, \EE))} r_U(\Def_{(\EE,U)}(A)),
\]
where $r_U: \Def_{(\EE,U) }\to \Def_\EE$ is the forgetful maps of functors. 
We call it the  functor of infinitesimal deformations  of $\EE$  with at least $k$ sections.
\end{definition}

As observed in  loc. cit., the functor $\Def^{k}_\EE $ is a functor of Artin rings, but unfortunately, it is not a deformation functor. This makes the  analysis of it more difficult: its first order deformations do not necessary form a vector space, and its obstructions do not have a clear meaning in term of the corresponding moduli space.

Here we get the following result, that is a version of \cite[Theorem 5.3]{DE Brill-Noether} over any algebraically closed field  $\mathbb{K}$ of characteristic zero.

\begin{theorem}\label{teorema tangente Def^r}
Let $X$ be a smooth projective variety over
an algebraically closed field  $\mathbb{K}$ of characteristic zero, such that $H^i(X,\OO_X)=0$ for $i=1,2$.
If $h^0(X, \EE)=k$, then the tangent space to the deformation functor $\Def_\EE^k$ is 
\[  t_{\Def^k_\EE} = \Def_\EE^k(\kk[\epsilon]) = \{ a \in H^1(X,\EE nd (\EE)) \ \mid \ a \cup s =0, \ \forall  \ s \in H^0(X,\EE) \}. \]  
If instead $h^0(X, \EE) \geq k+1$,  then 
the first order deformations of $\EE$ with at least $k$ sections are described by the cone
\[ \Def_{\EE}^k(\kk[\epsilon]) = \{\nu \in H^1(X, \EE nd (\EE)) \mid \exists U \in {Gr(k,H^0(X, \EE))}  \mbox{ such that } \nu \cup s =0, \forall s \in U\} \]
and the vector space generated by it, that we call the tangent space to $\Def_\EE^k$, is 
\[ t_{\Def_{\EE}^k} = H^1(X, \EE nd (\EE)). \]
\end{theorem}

\begin{proof}
In the case $h^0(X, \EE)=k$, the functor  $\Def^k_\EE$  is in one-to-on e correspondence with the functor $\Def_{(\EE,H^0(\EE))}$ and the tangent space is described in Corollary \ref{Cor. TDef(E,H0(E))}  as
\[t_{\Def^k_\EE}\cong t_{\Def_{(\EE,H^0(\EE))}} = \{ a \in H^1(X,\EE nd (\EE)) \ \mid \ a \cup s =0, \ \forall  \ s \in H^0(X,\EE) \}. \]  
If $h^0(X, \EE) \geq k+1$, by definition, 
\[\Def_\EE^k(\kk[\epsilon]) = \bigcup_{U \in Gr(k,H^0(X, \EE))} r_U(\Def_{(\EE,U)}(\kk[\epsilon])).\]
For each $U \in  Gr(k,H^0(X, \EE))$, we calculate the image  via $r_U$  of the tangent space to   the infinitesimal deformations of the pair $(\EE,U)$ using the exact sequence
\eqref{eq.theorem} 
\[ \cdots \to H^1(\Tot(H(\VV)^\Delta)) \stackrel{r_U}{\to} H^1(X,\EE nd (\EE))    \stackrel{\alpha}{\to}  \Hom (U, H^{1}(X,\EE)) \cdots . \]
Thus
\[  r_U(\Def_{(\EE,U)}(\kk[\epsilon]))= \ker \alpha = \{ \nu \in H^1(X, \EE nd (\EE)) \mid \nu \cup s =0, \forall s \in U\}  \]
and the first statement is proved. 

The second statement follows from a classical linear algebra argument, that can be found for example in \cite[Proposition 4.2]{ACGH}.  
\end{proof}

\begin{remark} 
We notice that our explicit description of the tangent space is also a particular case of the one of the Zariski tangent space to the cohomology jump functors done in \cite[Proposition 2.7]{Budur} using dgl pairs.
\end{remark}

Finally, we prove a new version of \cite[Propositions 5.5-5.6]{DE Brill-Noether} that concerns the smoothness of the functor $\Def_\EE^k$.

\begin{proposition} \label{prop.alpha sur equivalenza di liscezze}
Let $X$ be a smooth projective variety over
an algebraically closed field  $\mathbb{K}$ of characteristic zero, such that $H^i(X,\OO_X)=0$ for $i=1,2$.
  If there exists an $U \in  Gr(k,H^0(X, \EE))$ such that $\Hom(U, H^1(X, \EE))=0$ or, in an equivalent way, such that the map $\alpha:H^1(X, \EE nd (\EE)) \to \Hom(U, H^1(X,\EE)) $ in sequence \eqref{eq.theorem} is surjective, then
\[ \Def_\EE \mbox{ is smooth } \Leftrightarrow \Def_{(\EE,U)} \mbox{ is smooth } \Leftrightarrow \Def^k_\EE \mbox{ is smooth. } \]
\end{proposition}
\begin{proof}
From Corollary \ref{corollario alpha su implica r liscio}, the two equivalent hypotheses imply that the forgetful morphism $r_U$ is smooth. Then, the first equivalence is obvious.  
As regard the second equivalence, since the obstruction is complete, each $\EE_A \in \Def_{\EE}^k(A)$ comes from a pair $(\EE_A, i_A) \in \Def_{(\EE,U)}(A)$, for every $A\in \Art_\kk$. The above argument obviously implies  the equivalence between the smoothness of $\Def_{(\EE,U)}$ and $\Def_\EE^k$ . 
\end{proof}

\begin{proposition}  \label{prop.H2=0 liscezza}
Let $X$ be a smooth projective variety over
an algebraically closed field  $\mathbb{K}$ of characteristic zero, such that $H^i(X,\OO_X)=0$ for $i=1,2$.
If there exists an $U \in  Gr(k,H^0(X, \EE))$ such that $H^2(\Tot(H(\VV)^\Delta))=0$, then both the functors $\Def_{(\EE,U)}$ and $\Def_\EE^k$ are smooth. 
\end{proposition}
\begin{proof}
Since $H^2(\Tot(H(\VV)^\Delta))=0$, the functor $\Def_{(\EE,U)}$ is smooth and the relative obstruction to $r_U$ is zero, thus $r_U$ is smooth too.
These two properties assure that $\Def_\EE^k$ is smooth too. 
\end{proof}

\end{document}